\documentclass[a4paper,oneside,11pt]{article}
\usepackage{amsmath}
\usepackage{amsfonts}
\usepackage{amssymb}
\usepackage{titlesec}
\usepackage{graphicx}
\usepackage{subfigure}
\usepackage{soul}
\usepackage{fancyhdr}
\usepackage{float}
\usepackage{latexsym}
\usepackage{mathrsfs}
\usepackage{mathtools}
\usepackage{amsthm}
\usepackage{wasysym}
\usepackage{cite}
\usepackage{color}
\usepackage{geometry}
\usepackage{extarrows}
\usepackage{enumerate}
\usepackage{dsfont}
\usepackage{bm}
\usepackage{enumitem}

\DeclareMathOperator{\TV}{TV}
\DeclareMathOperator{\meas}{meas}

\usepackage[marginal]{footmisc}
\usepackage[dvips]{epsfig}
\usepackage{amscd}
\usepackage[bookmarks,colorlinks,citecolor=blue,linkcolor=red]{hyperref}
\usepackage{esint}    
\usepackage{stmaryrd}  
\usepackage{tikz}      

\allowdisplaybreaks

\newcommand{\R}{\mathbb{R}}

\renewcommand{\chi}{\mathbf{1}}
\newcommand{\pv}{\operatorname{PV}}
\numberwithin{equation}{section}

\newtheorem{definition}{Definition}[section]
\newtheorem{theorem}{Theorem}[section]
\newtheorem{lemma}{Lemma}[section]
\newtheorem{remark}{Remark}[section]
\newtheorem{proposition}{Proposition}[section]

\title{Steady-State Phase Transition States and Large-Time Asymptotics for the Cauchy Problem of Compressible Navier-Stokes Equations with van der Waals Equation of State}
\date{  }
\author{Yazhou C{\small HEN}$^1$,  Yi P{\small ENG}$^1$, Xiaoding S{\small HI}$^{1*}$, Xiaoping W{\small ANG}$^2$ \\[3mm]
\scriptsize$^{1}$ {College of Mathematics and Physics, Beijing University of
Chemical Technology, Beijing 100029, China}\\
\scriptsize$^2$ {School of Science and Engineering, The Chinese University of Hong Kong, Shenzhen, Shenzhen 518055, China  }
}

\begin{document}
\maketitle
\renewcommand{\thefootnote}{\fnsymbol{footnote}}
\footnotetext[1]{{Corresponding author. }\\
{Email: chenyz@mail.buct.edu.cn (Y.Chen), apengyi@163.com (Y.Peng), shixd@mail.buct.edu.cn (X.Shi), wangxiaoping@cuhk.edu.cn (X.\ Wang)}}

\begin{abstract}
We study the Cauchy problem for the one-dimensional isothermal compressible Navier--Stokes system with the van der Waals equation of state, and analyze the well-posedness and asymptotic stability of steady-state phase-transition solutions at subcritical temperatures. Working in Lagrangian coordinates, we use variational techniques to construct piecewise smooth steady-state solutions corresponding to two-phase equilibrium states under finite-mass perturbations. Introducing artificial-viscosity regularization, we first establish existence of the global minimizer of the regularized variational problem, unique within each fixed symmetry class; passing to the vanishing artificial-viscosity limit then yields a physically admissible steady-state solution with two phase interfaces. For initial perturbations that are small with respect to energy, piecewise-gradient, total-variation, and integral norms but may exhibit large pointwise jumps across phase boundaries, we derive uniform a priori bounds for weak solutions near the two-phase equilibrium profile. We further prove that these global weak solutions converge to the steady-state two-phase equilibrium in $L^\infty(\mathbb{R})$ as $t\to\infty$, thereby characterizing their large-time asymptotic behavior. Physically, finite mass perturbations break the uniqueness of the uniform steady state and generate infinitely many piecewise constant two-phase equilibria satisfying the Maxwell equal-area rule. Artificial viscosity acts as a selection principle that singles out the physically relevant two-interface steady-state configuration, and this two-phase equilibrium is asymptotically stable under small integral perturbations that allow large-amplitude discontinuities at phase interfaces.
\end{abstract}


\


\section{Introduction and the Main Results}\label{sec:int}

\hspace{2em} Gas--liquid phase transition is a fundamental thermodynamic phenomenon, corresponding to the interconversion between gaseous and liquid phases under prescribed conditions of temperature and pressure. Its underlying mechanism is governed by the competition between intermolecular cohesive forces and molecular thermal motion: thermodynamically, vaporization is an endothermic process, whereas condensation releases latent heat. Contemporary mathematical investigations in this field address the mechanisms of phase transition, the dynamical evolution of phase interfaces, and multi-physics coupling effects. For compressible fluids, gas--liquid phase transitions described by the van der Waals equation of state are closely connected to the well-posedness theory of the compressible Navier--Stokes equations, to spinodal instability analysis, and to the evolution of two-phase interfaces. Van der Waals fluids are therefore a central subject at the intersection of thermodynamics, fluid mechanics, and the mathematical theory of partial differential equations.

In contrast to the ideal gas equation of state, which models gas molecules as point particles with zero volume and no intermolecular interactions, the van der Waals equation of state, first proposed by J. D. van der Waals in \cite{vdW1894}, introduces two corrections that characterize the behavior of real non-ideal gases. The first is a volume correction, which accounts for the finite volume of gas molecules and reduces the effective free volume available for molecular motion. The second is a pressure correction describing the mean-field intermolecular attractive interaction, which lowers the macroscopic pressure of a real fluid relative to the prediction of the ideal gas law. As a physically plausible and mathematically tractable constitutive model, the van der Waals equation captures the essential thermodynamic features of gas--liquid two-phase flows. Its explicit form reads
\begin{equation}\label{eq:vdw-eos}
\left(p+\frac{a}{v^2}\right)(v-b)=R\theta.
\end{equation}
Here, $\theta>0$ denotes the absolute temperature of the fluid, $v$ denotes the specific volume (volume per unit mass), $R$ is the specific gas constant, and $a>0$ and $b>0$ are positive constants accounting for intermolecular cohesive forces and the finite volume of molecules, respectively.

According to the van der Waals equation of state \eqref{eq:vdw-eos}, a critical temperature $\theta_c$ can be uniquely determined. For temperatures satisfying $0 < \theta < \theta_c$, each $p$--$v$ isotherm possesses a nonmonotonic branch containing one local maximum and one local minimum. This nonmonotonic feature corresponds to mechanical instability, and the associated domain describes the range of liquid-gas phase coexistence; see Figure 1. The critical temperature reads
\begin{equation}\label{criticaltemperature}
\theta_c = \frac{8a}{27Rb},
\end{equation}
which sets the upper limit of the unstable two-phase region.
For physically meaningful states, pressure must stay positive for all admissible specific volumes $v>b$. Direct computation shows that a lower temperature bound $\theta_{*} = \frac a{4bR}$ is required to guarantee $p(v) > 0$ over $v>b$. In what follows, we restrict our analysis to isothermal flows with
\begin{equation}\label{tildetheta}
\frac{a}{4bR}=\theta_{*} < \tilde{\theta} < \theta_c=\frac{8a}{27Rb},
\end{equation}
where both inequalities are necessary for the existence of physically valid liquid-gas phase transitions. For a fixed temperature $\tilde{\theta}$ within this interval, the pressure and its first and second derivatives with respect to $v$ take the form
\begin{equation}\label{eq:p-derivatives}
p(v) = \frac{R\tilde{\theta}}{v-b} - \frac{a}{v^2}, \qquad p'(v) = -\frac{R\tilde{\theta}}{(v-b)^2} + \frac{2a}{v^3}, \qquad p''(v) = \frac{2R\tilde{\theta}}{(v-b)^3} - \frac{6a}{v^4}.
\end{equation}
\begin{figure}[htbp]
\centering
\includegraphics[width=0.7\textwidth]{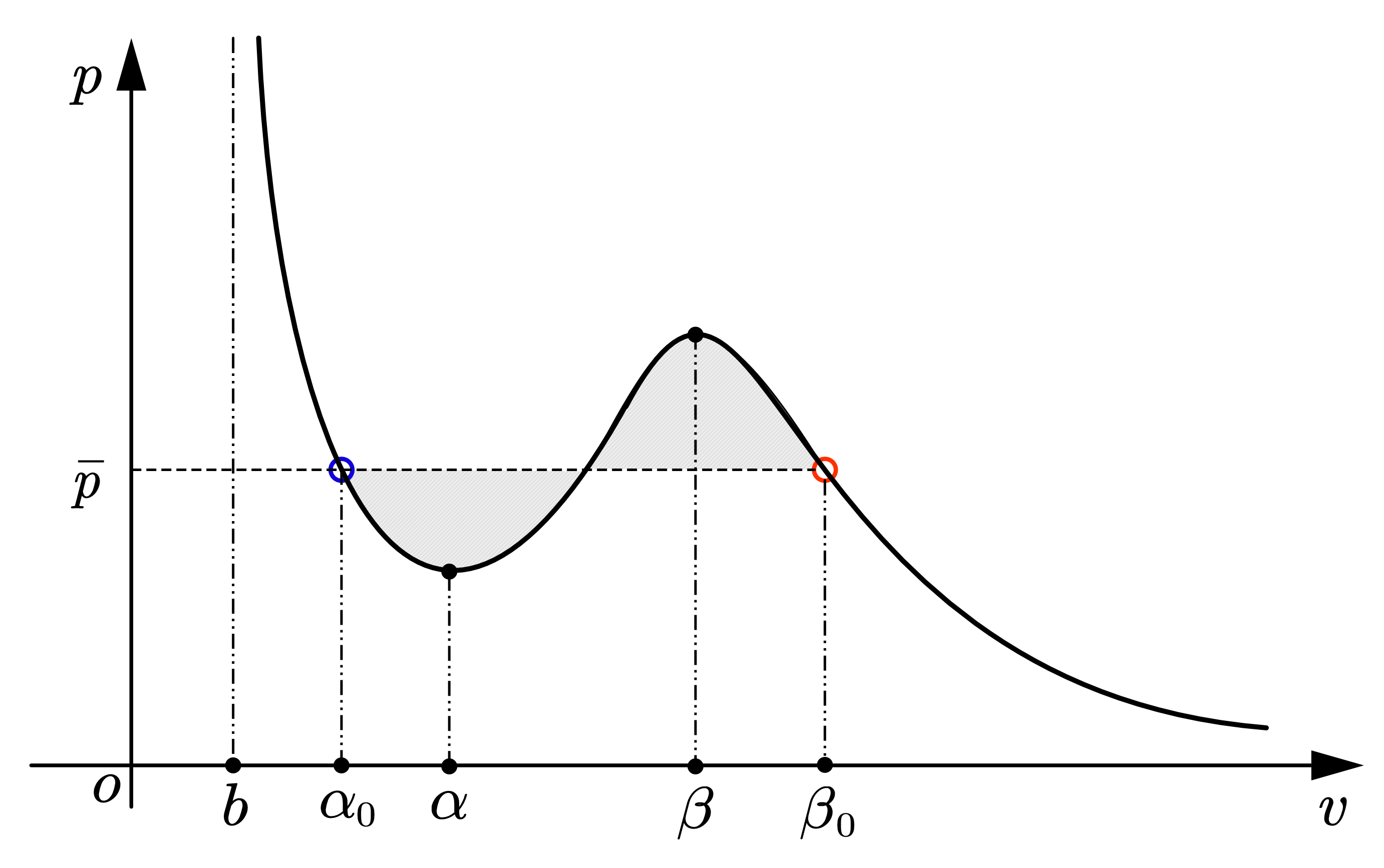}
\caption{van der Waals $p$-$v$ isotherm}
\label{fig:vdw-pv}
\end{figure}

The relation \(c^2 = -v^2 p'(v)\) directly characterizes the local wave behavior of the fluid system. When \(p'(v) > 0\), the squared sound speed becomes negative, leading to purely imaginary sound velocities. Mathematically, this property renders the Lagrangian governing system elliptic and hence the corresponding initial-value problem ill-posed. This mechanical instability is consistent with thermodynamic instability: the condition \(p'(v) > 0\) is equivalent to \(\partial_{vv} F < 0\) for the Helmholtz free energy \(F(v)\), whose convexity in the specific volume \(v\) is required for thermodynamic equilibrium stability. Uniform single-phase states can therefore not remain stable within this regime, which necessitates the Maxwell construction.

For any fixed subcritical temperature \(\tilde{\theta}\in(\theta_*,\theta_c)\), the pressure function \(p(v)\) admits two distinct spinodal points \(\alpha\) and \(\beta\). The interval \((\alpha,\beta)\) corresponds precisely to the elliptic and thermodynamically unstable domain, where continuous solution branches of the van der Waals equation fail physical feasibility. To eliminate this unphysical behavior and provide a rigorous description of liquid--vapor coexistence and phase transitions for van der Waals fluids, we introduce a pair of Maxwell equilibrium states \(\alpha_0\) and \(\beta_0\) satisfying the ordering \(b<\alpha_0<\alpha<\beta<\beta_0\), which follows from Maxwell's equal-area rule \cite{M1875}. These two coexisting phase states provide the physically admissible reference states for our subsequent phase-transition analysis; their rigorous definition is given below.
\begin{definition}\label{Maxwell}
\textbf{(Maxwell Equilibrium States and Maxwell Region \cite{M1875})}
For a fixed temperature \(\theta_* < \tilde\theta < \theta_c\), the Maxwell equilibrium states are uniquely specified via Maxwell's equal-area rule (see Figure 1), whose mathematical formulation reads
\begin{equation}\label{eq:Maxwell_equal_area}
\int_{\alpha_0}^{\beta_0} p(v) \, \mathrm{d}v = p(\alpha_0)(\beta_0 - \alpha_0) = p(\beta_0)(\beta_0 - \alpha_0)=\sigma_0(\beta_0 - \alpha_0).
\end{equation}
The parameters \(\alpha_0\) and \(\beta_0\) represent the specific volumes of the two thermodynamically coexisting phases, referred to as Maxwell equilibrium states, and \(\sigma_0:=p(\alpha_0)=p(\beta_0)\) denotes the common Maxwell saturation pressure. Correspondingly, we define the open interval
\begin{equation}\label{eq:Maxwell_region}
\Omega_{\mathrm{Maxwell}} = (\alpha_0, \beta_0)
\end{equation}
as the Maxwell region for the phase transition system.
\end{definition}
Under Definition \ref{Maxwell}, elementary calculations confirm that for intermediate subcritical isotherms with $\theta_* < \tilde{\theta} < \theta_c$, all characteristic specific volumes satisfy the strict ordering
\[
b < v_1 < \alpha_0 < \alpha < \beta < \beta_0 < v_2,
\]
where $v_1$ and $v_2$ stand for the inflection specific volumes obtained by solving $p''(v) = 0$.
These inflection points $v_1$, $v_2$ possess only geometric significance: they mark the locations where $p''(v)$ changes sign, with no intrinsic thermodynamic connection to two-phase coexistence.

The Maxwell region characterizes the physically valid regime for two-phase coexistence of van der Waals fluids. Fluid states evolving within this interval cannot maintain stable homogeneous single-phase configurations and will spontaneously undergo liquid-vapor phase separation. Against this physical background, the present work verifies mathematically that this interval serves as the intrinsic domain for phase-transition evolution. All admissible discontinuous phase-jump solutions constructed in this study exactly connect the two Maxwell equilibrium states.
\vspace{3mm}
\begin{definition}\label{phase-transition-def}
\textbf{(Phase Transition)} For compressible fluids governed by the subcritical van der Waals equation of state, phase transition (or phase separation) refers to the spontaneous dynamical process through which a uniform or slightly perturbed fluid system evolves into a stable two-phase coexistence state. Let \(\rho\) denote the mass density and \(v\) the specific volume, which satisfy the reciprocal relation \(\rho = 1/v\). The resulting coexistence state comprises two distinct stable equilibrium density phases: a low-density gaseous phase with density \(\beta_0^{-1}\) and a high-density liquid phase with density \(\alpha_0^{-1}\). The volumes \(\alpha_0\) and \(\beta_0\) correspond to stable Maxwell equilibrium states that bound the coexistence interval \((\alpha_0,\beta_0)\), whereas the inner interval \((\alpha,\beta)\) corresponds to the unstable spinodal regime where the inequality \(p'(v) > 0\) holds.
\end{definition}
\vspace{3mm}
This paper investigates the flow characteristics of van der Waals fluids over the full spatial domain. The primary objective is to identify all physically feasible equilibrium solutions related to phase transitions and examine their dynamic stability, thereby elucidating the mechanisms governing fluid phase transition behaviors. The analysis begins with a state classification based on the van der Waals equation of state. This classification enables a rigorous partitioning of the state space to distinguish stable, metastable, and unstable flow regimes, which provides a solid foundation for subsequent investigations into phase-separated structures and solutions with discontinuous interfacial properties.

For fixed $\tilde{\theta} \in (\theta_*, \theta_c)$, we decompose the specific-volume domain $(b, +\infty)$ into three subregions according to the van der Waals pressure relation \eqref{eq:vdw-eos}, corresponding to unstable, metastable, and stable thermodynamic states, as sketched in Figure 1. The detailed partition is defined as follows:
\begin{equation}\label{ABC}
\Omega_{\mathrm{unstable}} \triangleq (\alpha,\beta),\quad \Omega_{\mathrm{metastable}} \triangleq (\alpha_0,\alpha]\cup[\beta,\beta_0), \quad
\Omega_{\mathrm{stable}} \triangleq (b,\alpha_0]\cup[\beta_0,+\infty).
\end{equation}

With the above thermodynamic preparations in place, we now introduce the mathematical model investigated throughout this paper. We work with the one-dimensional compressible isothermal Navier-Stokes system for van der Waals fluids formulated in Lagrangian coordinates. This system describes time-dependent fluid motion and the dynamical evolution of phase transitions. The governing equations take the form
\begin{equation}\label{ns-lagrange}
\left\{
\begin{aligned}
& v_{t} - u_{x} = 0, & (x,t)\in\mathbb{R}\times\mathbb{R}_+,\\
& u_{t} + p_{x}(v) = \epsilon\left(\frac{u_{x}}{v}\right)_{x}, & (x,t)\in\mathbb{R}\times\mathbb{R}_+,\\
& (v,u)(x,0) = \big(v_0(x),u_0(x)\big) \to (v_{\infty},0), & \text{as } x\to\pm\infty,
\end{aligned}
\right.
\end{equation}
where $v$, $u$, and $p$ denote the specific volume, fluid velocity, and pressure, respectively, and $\epsilon$ represents the viscosity coefficient. Here the pressure function $p(v)$ obeys the subcritical van der Waals equation of state \eqref{eq:vdw-eos}. As indicated by
\begin{equation}\label{v-infty}
 v_\infty \in \{\alpha_0,\beta_0\},
\end{equation}
the far-field limiting specific volume $v_\infty$ is restricted to one of the two stable Maxwell equilibrium states $\alpha_0$ or $\beta_0$.

In addition, we impose a finite-mass constraint on the initial specific volume to ensure spatially localized perturbations:
\begin{equation}\label{localized perturbation condition}
\int_{-\infty}^{+\infty} \big( v_0(x) - v_\infty \big) \mathrm{d}x\triangleq M \in (-\infty,0)\cup(0,+\infty),
\end{equation}
This condition guarantees that the initial perturbation has finite total mass and is spatially localized in an integral sense.

System \eqref{ns-lagrange} denotes the one-dimensional compressible Navier--Stokes system formulated in Lagrangian coordinates. Its well-posedness and large-time asymptotic behavior have been extensively studied over the past several decades, most thoroughly for ideal gases obeying the relation $p = R\theta/v$, where $R$ is the specific gas constant. Matsumura and Nishida \cite{MN1980} first proved global well-posedness of strong solutions under small initial data; Lions \cite{Lions1998} later developed the renormalized weak solution framework, and Huang et al. \cite{HLX2012} constructed finite-energy solutions that allow for arbitrarily large density oscillations.

For the one-dimensional non-isentropic setting, global existence without small-data restrictions was first established by Kanel' \cite{Kanel1979}, Kazhikhov and Shelukhin \cite{KS1977}, and Kazhikhov \cite{Kazhikhov1982}, and later generalized by Jiang \cite{Jiang1999} and Li and Liang \cite{LL2016}. Stability results for rarefaction waves and shock waves were obtained by Matsumura and Nishihara \cite{MN1985,MN1986}, and further extended to contact discontinuities by Huang, Matsumura and Shi \cite{HMS2004}.

The van der Waals equation of state \eqref{eq:vdw-eos} gives a canonical description of pressure behavior for non-ideal fluids undergoing gas--liquid phase transitions. Unlike the ideal gas law, whose pressure is a strictly convex, monotone function of the specific volume, the van der Waals equation captures the inherently non-convex, non-monotone character of real fluids. This structural difference means classical analytical frameworks built for ideal-gas flows do not extend directly to general non-ideal systems.

Early theoretical work on van der Waals phase transitions focused on inviscid Riemann problems for the Lagrangian Navier--Stokes system \eqref{ns-lagrange}. James \cite{J1980}, Shearer \cite{S1986}, and Hattori \cite{H1986} constructed solutions under the entropy rate admissibility criterion. To clarify the wave structure of one-dimensional inviscid Riemann solutions, Hsiao \cite{H1990} and Yanagi \cite{Y1992} derived complete families of admissible shock and rarefaction wave curves.

Subsequent work extended these theories to viscous, capillary, and non-isothermal regimes. A closely related line of research adopts scale-invariant artificial viscosity regularization to treat the mixed-type phase-transition Riemann problem. In a two-part investigation, Slemrod \cite{S1989} first proved existence of regularized self-similar solutions, and Fan \cite{F1991} later completed the well-posedness argument for the vanishing-viscosity limit and gave a full characterization of the two-wave-fan structure of the limit solutions. In a related direction with physical viscosity and capillarity rather than artificial viscosity, Joseph and LeFloch \cite{JosephLeFloch2007} established uniform total-variation bounds and strong $L^1$ convergence of self-similar solutions to the phase-dynamics Riemann problem under large data.

Slemrod \cite{S1983,S1984} and Grinfeld \cite{G1989} applied Conley index theory to establish existence of traveling wave solutions in the Lagrangian framework. For multidimensional phase transition theory, we refer the reader to the comprehensive treatments of Benzoni-Gavage \cite{BG1998,BG1999} and Wang and Xin \cite{WX2003}, and the references cited therein. Zhang \cite{Z2007} used center-manifold analysis to prove existence of viscous--capillary traveling profiles for non-isothermal subsonic phase transitions in van der Waals fluids, covering the physically relevant regime of small viscosity, high thermal conductivity, and large heat capacity. Eden, Milani, and Nicolaenko \cite{EMN1993} studied periodic van der Waals Navier--Stokes equations with Korteweg capillary terms via semigroup theory, establishing weak stability of trivial steady states under small initial perturbations and proving existence of finite-dimensional local exponential attractors. Zumbrun \cite{Z2000} subsequently proved linear orbital stability of traveling wavefronts for viscous--capillary $p$-systems. Hoff \cite{Hoff1991,Hoff1992} and Hoff and Khodja \cite{HK1993} examined dynamic stability of steady weak solutions on the full real line for small initial disturbances.

We note that non-monotone pressure--density relations of similar form also arise in dynamic elastic bar models with $p(\rho)=\rho^{-3}-\rho^{-1}$. For one-dimensional Navier--Stokes equations regularized by artificial viscosity, Mei, Liu, and Wong \cite{MLW2007a,MLW2007b,MW2008} established global existence of strong solutions. More recently, Huang, Wang, and Zhang \cite{HWZ2025} investigated the large-time asymptotic behavior of solutions to one-dimensional cubic non-convex scalar viscous conservation laws, which share a comparable non-monotone structure with van der Waals fluid models. For the Navier--Stokes equations with non-monotone pressure, Tsyganov \cite{Tsyganov2007} studied the asymptotic behavior and dynamic stability of static phase mixtures, showing that fairly arbitrary discontinuous phase mixtures can arise as time-asymptotic limits of smooth evolutions. Most closely related to the present work, the authors and their collaborators \cite{CPSW2026} obtained global unique strong solutions to the one-dimensional Cauchy problem for a compressible non-isentropic Navier--Stokes/Allen--Cahn two-phase flow system with van der Waals equation of state and degenerate thermal conductivity. Working with non-vacuum, finite-temperature initial data and no smallness restriction on the data, they used a refined $L^2$-energy framework to show that all physical quantities remain uniformly bounded over any finite time interval, even when large density fluctuations and phase transitions are driven by the non-monotone pressure--volume relation.

In parallel with these analytical developments, numerical studies have advanced computational modeling of van der Waals fluid phase transitions. Affouf and Caflisch \cite{AC1991} carried out numerical simulations of the Cauchy problem under combined viscous and capillary regularization. Using a pseudospectral method with artificial viscosity, Hsieh and Wang \cite{HW1997} confirmed numerically that non-monotone pressure relations drive phase separation, and attributed phase boundary formation to unstable flow regimes corresponding to the positively sloped branch of van der Waals isotherms. They also observed that density-dependent artificial viscosity suppresses abrupt density jumps, smooths steep gradients, and effectively introduces interfacial energy effects into the system. In later work, He, Liu, and Shi \cite{HLS2018} revisited the regularized van der Waals fluid model using a second-order TVD Runge--Kutta splitting algorithm combined with the Jin--Xin relaxation scheme. This numerical framework was further extended by He and Shi to coupled Navier--Stokes/Cahn--Hilliard and Navier--Stokes/Allen--Cahn systems in subsequent work \cite{HS2020,HS2021}.

Beyond dynamical wave problems, the equilibrium phase structure of van der Waals fluids has also received careful theoretical attention. Carr, Gurtin, and Slemrod \cite{CGS1984} initiated the analysis of one-dimensional van der Waals--Cahn--Hilliard models on finite intervals with small gradient energy parameters, showing that monotone single-interface solutions and their symmetric counterparts are the only global energy minimizers. By contrast, multi-interface and non-monotone solutions carry higher energy and do not form local minimizers, which removes the energy degeneracy of piecewise constant two-phase states in the vanishing-gradient limit. The phase-plane structure of steady states and heteroclinic orbits of regularized van der Waals models, together with their connection with the Maxwell construction, was analyzed by Tang \cite{Tang1998}, while Fife and Wang \cite{FifeWang1998} established the existence of periodic structures with sharp layer profiles for small regularization parameters and explored their relation with the Cahn--Hilliard equation. As a more recent result, Chen et al. \cite{CHNPS2026} derived sharp phase-transition criteria for periodic steady isentropic van der Waals Navier--Stokes flows via artificial viscosity regularization. Specifically, liquid--vapor phase separation occurs if and only if the spatially averaged specific volume falls within the Maxwell coexistence interval. In the vanishing-viscosity limit, regularized solutions converge to sharp-interface Maxwell equilibria for averaged states inside this interval, and remain uniformly single-phase otherwise.

Inspired by these existing works, we analyze the structure of steady-state solutions and exploit the associated steady-state profiles to characterize the large-time dynamics of solutions to the Cauchy problem for the isothermal Navier-Stokes equations of van der Waals type. In particular, the corresponding steady-state system for \eqref{ns-lagrange} reads as follows:
\begin{equation}\label{steady-state ns-lagrange}
\left\{
\begin{aligned}
& u_{x}=0,& x\in \mathbb{R},\\
& p_{x}(v)=\epsilon\left(\frac{u_{x}}{v}\right)_{x}, & x\in \mathbb{R}, \\
& \lim_{x\to\pm\infty} (v,u) = (v_{\infty},0) ,
\end{aligned}
\right.
\end{equation}
together with the following finite-mass perturbation condition:
\begin{equation}\label{law of mass conservation}
\int_{-\infty}^{+\infty}  (v(x) - v_\infty ) \mathrm{d}x =M\in (-\infty,0)\cup(0,+\infty).
\end{equation}
Substituting \eqref{steady-state ns-lagrange}$_1$ into \eqref{steady-state ns-lagrange}$_2$ immediately yields
\begin{equation}\label{stationary-equation}
 \left\{\begin{aligned}
 & - p_{x}(v)=0,& x\in \mathbb{R}, \\
 & \lim_{x\to\pm\infty} v =v_{\infty}.
\end{aligned}\right.
\end{equation}

Employing variational methods, we establish the existence of piecewise smooth solutions to the steady-state system \eqref{stationary-equation} subject to the finite-mass perturbation condition \eqref{law of mass conservation}. Such solutions carry distinct physical significance within van der Waals fluid dynamics and accurately capture fluid phase transitions. The precise statement of this result is provided in the following theorem:
\vspace{3mm}

\begin{theorem}\label{thm:existence_pw_solution}
Consider the steady-state system \eqref{stationary-equation} subject to the mass conservation constraint \eqref{law of mass conservation}. A unique constant steady-state solution $v \equiv v_\infty$ exists if and only if the mass perturbation parameter vanishes, i.e., $M = 0$. In contrast, for any $M \neq 0$, the system admits infinitely many piecewise constant phase-transition solutions. These solutions take values exclusively in the Maxwell two-phase set $\{\alpha_0,\beta_0\}$ defined in Definition \ref{Maxwell} and satisfy the Maxwell equal-area rule.

In all such solutions, $v_\infty$ coincides with either $\alpha_0$ or $\beta_0$. The solution equals $v_\infty$ on $\mathbb{R}\setminus S$ and switches to the alternative phase value from $\{\alpha_0,\beta_0\}$ on $S$. Precisely,
\begin{equation}\label{phase-transition-solution}
v(x) \triangleq v_{\infty}\,\chi_{\mathbb{R}\setminus S}(x)+ v_{\mathrm{alt}} \chi_{S}(x),
\end{equation}
where  \(\chi_A(x)\) denotes the characteristic function of the set $A$. $v_\infty$ is determined by \eqref{v-infty}, and $v_{\mathrm{alt}}$ denotes the unique element in $\{\alpha_0,\beta_0\}\setminus\{v_\infty\}$, namely
\begin{equation}\label{valt}
v_{\mathrm{alt}}\in\{\alpha_0,\beta_0\}\setminus\{v_\infty\}.
\end{equation}
Here, $S$ and $\mathbb{R}\setminus S$ form a disjoint partition of $\mathbb{R}$ into Lebesgue measurable sets, with the measure of $S$ prescribed by
\begin{equation}\label{measure}
\mathrm{measure}(S)=\frac{|M|}{\beta_0-\alpha_0},\qquad L\triangleq\frac{|M|}{2(\beta_0-\alpha_0)}.
 \end{equation}
\end{theorem}

\vspace{3mm}

As shown in the preceding discussion, the steady-state system \eqref{stationary-equation} subject to the constraint \eqref{law of mass conservation} admits infinitely many solutions provided that $M \neq 0$. We now address the following question: 
\vspace{3mm}

\textit{From a physical perspective, does a physically preferable candidate exist for these  two-phase solutions?}
\vspace{3mm}

\noindent
In fact, the conclusion of Theorem \ref{thm:existence_pw_solution} indicates that the formation of interfaces (density jumps) does not lead to an increase in energy. We therefore expect that, within a theoretical framework incorporating interfacial energy, the energy-minimizing two-phase solution to the steady-state system \eqref{stationary-equation} is the two-interface solution, that is, the solution possessing the minimal number of phase interfaces:
\begin{equation}\label{vstar}
\mathcal{V}\triangleq v_{\mathrm{alt}} \chi_{I_*}(x) + v_{\infty} \chi_{\mathbb{R} \setminus I_*}(x),
\end{equation}
where \(I_*\subset\mathbb{R}\) denotes an interval with \(\operatorname{meas}I_*=2L\), and \(\chi_\Omega(x)\) stands for the characteristic function of the set \(\Omega\).

Motivated by the numerical investigations carried out by Hsieh and Wang \cite{HW1997}, we add the artificial-viscosity term $\tilde\varepsilon v v_{xx}$ to \eqref{steady-state ns-lagrange}, where $\tilde{\varepsilon}>0$ denotes the artificial-viscosity coefficient, independent of the physical viscosity $\epsilon$. Throughout the paper, the product ${\tilde\varepsilon\epsilon}=\epsilon\tilde{\varepsilon}$ is treated as a single total-viscosity coefficient, never expanded as a product; it plays the role of the (small) diffusion coefficient of the regularized problem and is sent to zero in the artificial-viscosity limit ${\tilde\varepsilon\epsilon}\to0^+$, while the physical viscosity $\epsilon$ is held fixed. The same symbol also appears as the superscript of the solution $v^{\tilde\varepsilon\epsilon}$ simply as a label for the regularized solution at total viscosity ${\tilde\varepsilon\epsilon}$. We then construct an approximate system with periodic boundary conditions
\begin{equation}\label{Approximate ns-lagrange}
\left\{\begin{array}{llll}
\displaystyle -u_{x}=\tilde\varepsilon vv_{xx},& x\in [-n,n],\\
\displaystyle p_{x}(v)=\epsilon\big(\frac{u_{x}}{v}\big)_{x},& x \in [-n,n], \\
\displaystyle  (v,u)(-n)=(v,u)(n), 
\end{array}\right.
\end{equation}
supplemented by the integral constraint that the mean value of $v$ on $[-n,n]$ be prescribed, according to the finite-mass condition \eqref{law of mass conservation}, as
\begin{equation}\label{barv}
v_{\mathrm{mean}}\triangleq \frac1{2n}\int_{-n}^nvdx= \frac Lnv_{\mathrm{alt}}+\frac{n-L}nv_\infty,\qquad n>L,\end{equation}
It is easy to check, using \eqref{barv} and the Maxwell construction of Definition \ref{Maxwell}, that $\alpha_0<v_{\mathrm{mean}}<\beta_0$.

We now investigate the approximate system \eqref{Approximate ns-lagrange} subject to the constraint \eqref{barv}. Unfortunately, recent work by Chen et al. \cite{CHNPS2026} demonstrates that solutions to this system \eqref{Approximate ns-lagrange} fail to be unique. A large portion of these solutions are non-physical and cannot be realized in practice. It is therefore necessary to single out physically relevant solutions from a physical perspective. This naturally motivates us to seek admissible solutions by minimizing the global energy, which leads us to introduce the energy functional defined below.
\begin{equation}\label{J-v}
J(v)\triangleq\int_{-n}^{n}\bigg(-\int_{v_{\mathrm{mean}}}^{v(x)}p(s)\,ds+\frac{\tilde\varepsilon\epsilon}{2}v_{x}^2+\sigma v(x)\bigg)\mathrm{d}x,\qquad v(x)\in H^1_{\mathrm{per}}([-n,n]).
\end{equation}
Here \(\sigma\) denotes the Lagrange multiplier. Subject to the  constraint \eqref{barv}, we are naturally led to look for global minimizers of the functional given below:
\begin{equation}\label{Lagrange-multiplier-J}
\min_{v\in H^1_{\mathrm{per}}([-n,n])} J(v),
\end{equation}
where \(H^l_{\mathrm{per}}([-n,n])\) denotes the periodic Sobolev space
\begin{equation}\label{periodic-sobolev-space}
H^l_{\mathrm{per}}([-n,n])\triangleq \Big\{f\in H^l_{\mathrm{loc}}(\mathbb{R})\ \Big|\ f(x+2n)=f(x),\ \forall x\in\mathbb{R}\Big\},\quad l\geq0,
\end{equation}
so that each element is a $2n$-periodic function with locally square-integrable derivatives of order up to $l$.

With the foregoing preliminaries in place, we are now ready to present the definition of admissible   solution for the steady-state system \eqref{steady-state ns-lagrange}  subject to the condition  \eqref{law of mass conservation}.
\vspace{3mm} 

\begin{definition}[Admissible Phase Transition Solution]
\label{def:AdmissibleSolution}
A function $\mathcal{V}$ is called an admissible phase transition solution to system \eqref{stationary-equation} under the mass conservation constraint \eqref{law of mass conservation} if it can be obtained via a two-step limiting process starting from periodic variational minimizers.
For all $n>L$ and sufficiently small $0<\tilde{\varepsilon}<\tilde{\varepsilon}_0$, let $v^{\tilde{\varepsilon},n}$ denote the unique global energy minimizer, within each fixed symmetry class (unimodal or anti-unimodal), of problem \eqref{Lagrange-multiplier-J} on $[-n,n]$ subject to periodic boundary conditions. As $n\to\infty$, $v^{\tilde{\varepsilon},n}$ converges locally uniformly on $\mathbb{R}$ to a smooth steady viscous solution defined on the full real line, denoted by $v^{\tilde{\varepsilon}}$. Taking $\tilde{\varepsilon}\to0^+$, some subsequence of $\{v^{\tilde{\varepsilon}}\}$ converges pointwise almost everywhere on $\mathbb{R}$ and strongly in $L^1_{\mathrm{loc}}(\mathbb{R})$ to $\mathcal{V}$.
\end{definition}

\vspace{3mm} 

\begin{remark}
Admissible solutions are piecewise constant two-phase functions taking only the equilibrium specific volumes of the van der Waals fluid. They satisfy the Rankine-Hugoniot pressure condition at steady-state phase interfaces, preserve the total excess mass, and are thermodynamically consistent with the Maxwell equal-area construction.
\end{remark}
\vspace{3mm} 

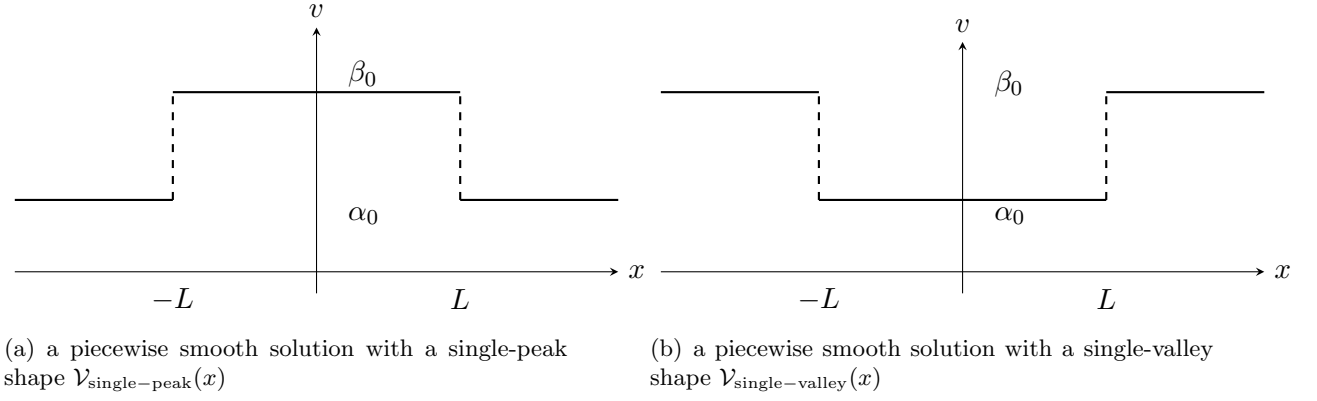
\begin{figure}[htbp]
    \centering
    \subfigure[a piecewise smooth solution with a single-peak shape $\mathcal{V}_{\mathrm{single-peak}}(x)$]
    {
        \begin{minipage}[t]{0.45\textwidth}
            \centering
            \begin{tikzpicture}[>=stealth, scale=0.95]
            \def\L{2.0}    
            \def\valpha{1}
            \def\vbeta{2.5}
            \def\xmax{4.2}

            \draw[->] (-\xmax,0) -- (\xmax,0) node[right] {$x$};
            \draw[->] (0,-0.3) -- (0,\vbeta+0.9) node[above] {$v$};

            \draw[thick] (-\xmax,\valpha) -- (-\L,\valpha);
            \draw[thick] (-\L,\vbeta) -- (\L,\vbeta);
            \draw[thick] (\L,\valpha) -- (\xmax,\valpha);

            \draw[dashed, thick] (-\L,\valpha) -- (-\L,\vbeta);
            \draw[dashed, thick] (\L,\valpha) -- (\L,\vbeta);

            \node at (-\L,-0.38) {$-L$};
            \node at (\L,-0.38) {$L$};
            \node[anchor=west] at (0.3,\valpha-0.22) {$\alpha_0$};
            \node[anchor=west] at (0.3,\vbeta+0.25) {$\beta_0$};
            \end{tikzpicture}
        \end{minipage}
    }
    \hfill
    \subfigure[a piecewise smooth solution with a single-valley shape $\mathcal{V}_{\mathrm{single-valley}}(x)$]
    {
        \begin{minipage}[t]{0.45\textwidth}
            \centering
            \begin{tikzpicture}[>=stealth, scale=0.95]
            \def\L{2.0}    
            \def\valpha{1}
            \def\vbeta{2.5}
            \def\xmax{4.2}

            \draw[->] (-\xmax,0) -- (\xmax,0) node[right] {$x$};
            \draw[->] (0,-0.3) -- (0,\vbeta+0.7) node[above] {$v$};

            \draw[thick] (-\xmax,\vbeta) -- (-\L,\vbeta);
            \draw[thick] (-\L,\valpha) -- (\L,\valpha);
            \draw[thick] (\L,\vbeta) -- (\xmax,\vbeta);

            \draw[dashed, thick] (-\L,\valpha) -- (-\L,\vbeta);
            \draw[dashed, thick] (\L,\valpha) -- (\L,\vbeta);

            \node at (-\L,-0.38) {$-L$};
            \node at (\L,-0.38) {$L$};
            \node[anchor=west] at (0.3,\valpha-0.22) {$\alpha_0$};
            \node[anchor=west] at (0.3,\vbeta+0.12) {$\beta_0$};
            \end{tikzpicture}
        \end{minipage}
    }
    \caption{phase transition solution with two distinct interfaces}
    \label{model2figure1}
\end{figure}

With the structural characterization of periodic phase-transition solutions established above, we may construct the following sequence of functions defined on the entire real line via extension. 
\begin{equation}\label{v-n}
\mathcal{V}^{(n)}(x) \triangleq v_{\mathrm{alt}} \chi_{I_n}(x) + v_{\infty} \chi_{\R \setminus I_n}(x),
\end{equation}
where \(\chi_\Omega(x)\) denotes the characteristic function of the set \(\Omega\), and \(\{I_n\}\) is a sequence of intervals in \(\R\) with \(\operatorname{meas}I_n=2L\).
\vspace{3mm}

\vspace{3mm}
\begin{theorem}\label{thm:weak-compactness}
Consider the steady-state system \eqref{steady-state ns-lagrange} supplemented with the mass conservation constraint \eqref{law of mass conservation}, where \(M\neq 0\).
\begin{enumerate}[label=(\alph*)]
\item
There exists a constant \(\tilde\varepsilon_0 > 0\) such that for every \(n>L\) and all \(0 < \tilde\varepsilon < \tilde\varepsilon_0\), the constrained system \eqref{Approximate ns-lagrange} under condition \eqref{barv} admits, within each fixed symmetry class (unimodal or anti-unimodal), a unique smooth solution \((v^{\tilde\varepsilon,n},0)\). Each \(v^{\tilde\varepsilon,n}\) is a global minimizer of the variational problem \eqref{Lagrange-multiplier-J}, and as \(\tilde\varepsilon \to 0^+\), \(v^{\tilde\varepsilon,n}(x)\) converges pointwise to a piecewise constant two-phase function \(\mathcal{V}^{(n)}(x)\) for all \(x\in[-n,n]\) except at the two discontinuity points.

\item
For each fixed \(0<\tilde\varepsilon<\tilde\varepsilon_0\), as \(n\to\infty\) the family \(\{v^{\tilde\varepsilon,n}\}\) converges locally uniformly on \(\mathbb{R}\) to a smooth steady viscous solution \(v^{\tilde\varepsilon}\) defined on the whole real line.
Furthermore, as \(\tilde\varepsilon\to 0^+\), there exists a subsequence of \(\{v^{\tilde\varepsilon}\}\) that converges pointwise almost everywhere on \(\mathbb{R}\), and strongly in \(L^1_{\mathrm{loc}}(\mathbb{R})\), to a piecewise constant function
\begin{equation*}
\mathcal{V}(x) \triangleq v_{\mathrm{alt}} \chi_{I_*}(x) + v_{\infty} \chi_{\mathbb{R} \setminus I_*}(x),
\end{equation*}
where \(I_*\subset\mathbb{R}\) is an interval with \(\mathrm{meas}(I_*)=2L\), and \(\chi_{I_*}(x)\) stands for the characteristic function of the set \({I_*}\).
The limit satisfies the Rankine-Hugoniot pressure condition at the steady-state phase interfaces, is thermodynamically consistent with the Maxwell equal-area construction, and is an admissible phase transition solution in the sense of Definition \ref{def:AdmissibleSolution}.
\end{enumerate}
\end{theorem}
\vspace{3mm}

\begin{remark}
Theorem \ref{thm:weak-compactness} rigorously constructs admissible phase-transition solutions for steady van der Waals fluids. The inviscid steady-state system possesses numerous weak solutions containing phase interfaces, yet many lack physical relevance. The viscous regularization in the approximate system imitates fluid dissipation, and the vanishing-viscosity limit implements a selection principle, analogous to the classical admissibility condition for shock waves.
Approximation on expanding bounded intervals $[-n,n]$ is a technical device to treat the unbounded domain $\mathbb{R}$. Physically, it reflects the localization of the phase transition layer: the transition region remains confined and does not diffuse as the computational domain expands, with uniform far-field fluid state away from the interface.
The resulting two-phase limit satisfies the Rankine--Hugoniot condition at steady-state interfaces and complies with the Maxwell equal-area rule. Consequently, the selected solution corresponds to the thermodynamic equilibrium state with minimal free energy, consistent with the variational structure of the model. This establishes a mathematically rigorous, thermodynamically consistent admissibility criterion for phase transitions in nonstrictly hyperbolic systems.
\end{remark}\vspace{3mm}

Without loss of generality, we translate the $x$-axis so that the discontinuities of $\mathcal{V}$ occur at $x=\pm L$. The profile of the limiting solution $\mathcal{V}$ takes the piecewise-constant form (see Figure 2)
\begin{align}
\mathcal{V}(x)\triangleq
\begin{cases}
v_\infty, & x\in I_{-\infty}\triangleq (-\infty,-L),\\
v_{\mathrm{alt}}, & x\in I_*\triangleq [-L,L],\\
v_\infty, & x\in I_{+\infty}\triangleq  (L,+\infty),
\end{cases}
\label{Phase-transition-solution}
\end{align}
where the parameter $L$ represents the width of the intermediate phase region and obeys
\begin{equation}\label{length-of-L}
L \triangleq  \frac{1}{2(\beta_0 - \alpha_0)}
\left|\int_{-\infty}^{+\infty} \big(v_0(x)-v_\infty\big)\mathrm{d}x\right|
= \frac{|M|}{2(\beta_0 - \alpha_0)}.
\end{equation}

Based on the admissible steady-state solutions to system \eqref{steady-state ns-lagrange}, we analyze the global well-posedness and large-time asymptotic stability of solutions to the corresponding time-dependent Navier--Stokes system \eqref{ns-lagrange}. The main difficulty comes from the non-monotonic, non-convex structure of the van der Waals pressure law \eqref{eq:vdw-eos}, which invalidates classical energy-entropy arguments for pointwise bound estimates and complicates the global existence and asymptotic analysis.

Accordingly, we impose initial data near the two Maxwell phase-equilibrium states $\alpha_0$ and $\beta_0$, which lie in the hyperbolic regimes $(b,\alpha)$ and $(\beta,+\infty)$ respectively, where $p'(v)<0$; the admissible range $[b+\underline{v},\bar{v}]$ in \eqref{K} nevertheless contains the elliptic unstable region $(\alpha,\beta)$ with $p'(v)>0$, which is precisely where the phase-transition layers develop. Phase transitions between different Maxwell states are achieved through jump discontinuities. For this reason, all solutions considered herein are formulated in the weak sense, as defined below.
\vspace{3mm}

\begin{definition}[Weak Solution]
Let $\mathbf{U}_0$ be a given initial function. We say that a vector-valued function $\mathbf{U}$ is a \emph{weak solution} to the parabolic system
\[
\mathbf{U}_t + \big(\mathbf{F}(\mathbf{U})\big)_x = \big(\mathbf{B}(\mathbf{U})\mathbf{U}_x\big)_x,\qquad \mathbf{U}(\cdot,0) = \mathbf{U}_0,
\]
if for all $t>0$ and all $\boldsymbol{\phi} \in \big(\mathscr{D}(\mathbb{R}^2)\big)^2$, there holds
\[
\begin{aligned}
&\int_{\mathbb{R}} \big(\mathbf{U}\cdot\boldsymbol{\phi}\big)(x,t)\,\mathrm{d}x - \int_{\mathbb{R}} \big(\mathbf{U}\cdot\boldsymbol{\phi}\big)(x,0)\,\mathrm{d}x \\
&\qquad -\int_{0}^{t}\int_{\mathbb{R}} \Big(\mathbf{U}\cdot\boldsymbol{\phi}_\tau + \mathbf{F}(\mathbf{U})\cdot\boldsymbol{\phi}_x\Big)\mathrm{d}x\mathrm{d}\tau
= -\int_{0}^{t}\int_{\mathbb{R}} \big(\mathbf{B}(\mathbf{U})\mathbf{U}_x\big)\cdot\boldsymbol{\phi}_x\,\mathrm{d}x\mathrm{d}\tau.
\end{aligned}
\]
\end{definition}
\vspace{3mm}

We derive hereafter the jump relations satisfied by discontinuous solutions. Let $(v,u)$ solve the  Navier-Stokes system \eqref{ns-lagrange}, and suppose this solution is smooth away from a discontinuity curve $x(t)$ propagating with speed $s = \mathrm{d}x/\mathrm{d}t$. Direct computation yields the following Rankine--Hugoniot conditions across the discontinuity:
\begin{equation}\label{rh-conditon1}
	-s\llbracket v\rrbracket=\llbracket u\rrbracket, \quad s\llbracket u\rrbracket=\llbracket p - \frac{\epsilon u_{x}}{v}\rrbracket,
\end{equation}
where $\llbracket\cdot\rrbracket$ denotes the standard jump operator across the discontinuity interface, defined as $\llbracket v\rrbracket = v(x^+, t) - v(x^-, t)$.
A classical result due to Hoff and Smoller \cite[Theorem 4.2]{HS1985} shows that the parabolic structure of the momentum equation in \eqref{ns-lagrange} instantaneously regularizes initial velocity discontinuities for all $t>0$, so that $\llbracket u\rrbracket = 0$. Substituting this identity into the jump conditions immediately gives $s=0$, which further reduces the second jump relation to
\begin{equation}\label{jump0}
  \bigl \llbracket p(v) - \frac{\epsilon u_{x}}{v}\bigr\rrbracket=0.
\end{equation}
As a result, any initial discontinuity in $v_0$ located at $x=y_i$ persists along the fixed particle path $x=y_i$ for all subsequent time, while the velocity field $u$ remains continuous for all $t>0$.    
Define $\Phi = \ln v$. Combining the system \eqref{ns-lagrange} with the simplified jump condition \eqref{jump0}, we arrive at
\begin{equation}\label{eq:jump-alpha0}
\bigl\llbracket \Phi \bigr\rrbracket_{t} = \left\llbracket \frac{u_{x}}{v} \right\rrbracket =\frac1{\epsilon} \bigl\llbracket p(v) \bigr\rrbracket = \frac1{\epsilon}\lambda_i(t)\bigl\llbracket \Phi \bigr\rrbracket,
\end{equation}
where the time-dependent coefficient $\lambda_i(t)$ is given by
\begin{equation}\label{eq:jump-alpha1}
\lambda_i(t)
= \dfrac{p(v^+)-p(v^-)}{\Phi(v^+)-\Phi(v^-)}
= \xi\, p'(\xi).
\end{equation}
Here $v^{\pm} = v(y_i^\pm ,t)$, and the intermediate value $\xi$ lies strictly between $v^-$ and $v^+$. Solving the resulting ordinary differential equation \eqref{eq:jump-alpha0} yields the explicit temporal evolution of the jump magnitude:
\begin{equation}\label{eq:jump-alpha2}
\llbracket \Phi(y_i,t) \rrbracket = \llbracket \Phi(y_i,0) \rrbracket \exp \left( \epsilon^{-1}\int_{0}^{t}\lambda_i (\tau) \mathrm{d}\tau \right).
\end{equation}
Note that $\lambda_i$ can be written as $\lambda_i = \llbracket p \rrbracket/\llbracket \Phi \rrbracket$. For a jump connecting the two Maxwell equilibrium states one has $p(\alpha_0)=p(\beta_0)$, hence $\llbracket p \rrbracket=0$ and therefore $\lambda_i\equiv 0$: Maxwell jumps are neutrally stable and persist for all time, which is consistent with the convergence of $v$ toward the discontinuous profile $\mathcal{V}$. Only jumps whose end states fail to satisfy the equal-area rule decay exponentially.

We now state the structural constraints imposed on the initial data. Since the initial profile is constructed as a perturbation of the phase-transition steady solution, the initial specific volume $v$ is allowed to admit large-amplitude variations. To characterize the admissible range of $v$, we introduce positive constants $\underline{v}$ and $\bar{v}$ satisfying the nested inclusion relation
\begin{equation}\label{K}
(\alpha_0,\beta_0)\subset\big[b+2\underline{v},\bar{v}-\underline{v}\big]\subset\mathbb{R}_+.
\end{equation}
For all $v\in [b+\underline{v}, \bar v]$, the derivatives associated with the van der Waals pressure function satisfy uniform boundedness
\begin{equation}\label{eq:vvv}
\begin{aligned}
-C^{-1}\leq p'(v)\leq C,\quad 
-C^{-1}\leq vp'(v)\leq C.
\end{aligned}	
\end{equation}

Combining the van der Waals pressure law \eqref{eq:vdw-eos} with the Maxwell point properties defined in Definition \ref{Maxwell}, we impose an additional structural assumption on the above bounded interval. There exist $\delta$-neighborhoods $U(\alpha_0,\delta)$ and $U(\beta_0,\delta)$ of the two Maxwell points such that
\begin{equation}\label{U-l,U-r}
U(\alpha_0,\delta),U(\beta_0,\delta) \subset \operatorname{int} \big[b+\underline{v}, \bar v\big],
\end{equation}
and the pressure function remains strictly decreasing within these neighborhoods:
\begin{equation}
\label{pressure in stable region}
p'(v) < 0 \quad \text{in } U(\alpha_0,\delta)\cup U(\beta_0,\delta).
\end{equation}

\noindent\textbf{\normalsize Notations.} Throughout this paper, we use the following notation for function spaces: $L^\infty(\mathbb{R})$ denotes the space of essentially bounded measurable functions on $\mathbb{R}$, endowed with the essential supremum norm
$
\|f\|_{L^\infty} \overset{\mathrm{def}}{=} \operatorname*{ess\,sup}_{x\in\mathbb{R}} |f(x)|.
$
For $q \geq 1$, $L^q(\mathbb{R})$ (abbreviated simply as $L^q$ when no ambiguity arises) is the space of $q$-th power integrable real-valued functions on $\mathbb{R}$, with norm
$
\|f\|_{L^q} \overset{\mathrm{def}}{=} \left( \int_{-\infty}^{+\infty} |f(x)|^q \, \mathrm{d}x \right)^{\frac{1}{q}}.
$
In the special case $q=2$, we write $\|f\|$ as a shorthand for $\|f\|_{L^2}$.
For a non-negative integer $l$, $H^l(\mathbb{R})$ (abbreviated as $H^l$ when the context is clear) denotes the Sobolev space of functions $f \in L^2(\mathbb{R})$ whose weak derivatives $\partial_x^j f$ for $j = 1, \dots, l$ are also square-integrable. The $H^l$ norm is defined by
$
\|f\|_l \overset{\mathrm{def}}{=} \left( \sum_{j=0}^l \|\partial_x^j f\|^2 \right)^{\frac{1}{2}},
$
where $\|\cdot\|$ stands for the $L^2$ norm introduced above.
 $\operatorname{Var}$ designates the  total variation, defined by 
$
\operatorname{Var} f
= \sup\{
\sum_{i=1}^{n} |f(x_i) - f(x_{i-1})|
\,\big|\,
-\infty<x_0<x_1<\dots<x_n<+\infty,\ n\in\mathbb{N}
\}$.
We further let $\|\cdot\|_{\text{pw}}$ denote the piecewise $L^2$ norm. Specifically, for all $t>0$, we define
\begin{equation}\label{eq:piecewise_L2_norm}
\|v_x(\cdot,t)\|^2_{\text{pw},L^2}
\triangleq \int_{-\infty}^{-L} v_x^2(x,t) \, \mathrm{d}x
+ \int_{-L}^{L} v_x^2(x,t) \, \mathrm{d}x
+ \int_{L}^{+\infty} v_x^2(x,t) \, \mathrm{d}x,
\end{equation}
which is well-defined provided that the piecewise discontinuous function $v(\cdot,t)$ satisfies
\[
v(\cdot,t)\in H^1(-\infty,-L) \cap H^1(-L,L) \cap H^1(L,+\infty).
\]
In addition, we adopt the standard notation for parabolic H\"older seminorms, defined as
\[
\langle w\rangle_\Omega^{a,b}
\triangleq \sup_{(x,t),\,(y,s)\in \Omega}
\frac{|w(x,t)-w(y,s)|}{|x-y|^a + |t-s|^b},
\]
for $a,b\in(0,1)$ and any domain $\Omega\subset\mathbb{R}\times\mathbb{R}_+$.
\vspace{3mm}

We now fix a constant $\vartheta \in (0,1)$ and set $\omega = \min\{1, t\}$; the symbol $\omega$ denotes the time-weight function (to be distinguished from the Lagrange multiplier $\sigma$ introduced in \eqref{J-v}). For solutions $(v,u)$ to system \eqref{ns-lagrange}, we introduce the following functionals:
\begin{subequations}\label{eq:energy-norms}
\begin{align}
A(t_0) &\triangleq \sup_{0\le t\le t_0} \left[ \|(v-\mathcal{V}, u)(\cdot,t)\|^2 + \|v_x(\cdot,t)\|_{\mathrm{pw},L^2}^2 \right]
+ \int_0^{t_0} \left( \|u_x(\cdot,t)\|^2 \right) \mathrm{d}t; \label{eq:norm-A} \\[6pt]
B(t_0) &\triangleq \sup_{0<t\le t_0}  \omega^{1/2}\|u_x(\cdot,t)\|^2   + \int_0^{t_0} \omega^{1/2+\vartheta} \left( \|u_t(\cdot,t)\|^2 + \left\|\left(\frac{\epsilon u_x}{v}\right)_x(\cdot,t)\right\|_{\text{pw},L^2}^2 \right) \mathrm{d}t \nonumber \\
&\quad + \int_0^{t_0} \omega  \|(u_x(\cdot,t))^2\|^2   \mathrm{d}t; \label{eq:norm-B} \\[6pt]
F(t_0) &\triangleq \sup_{0<t\le t_0}  \omega^{3/2} \left( \|u_t(\cdot,t)\|^2 + \left\|\left(\frac{\epsilon u_x}{v}\right)_x(\cdot,t)\right\|_{\text{pw},L^2}^2 \right)  + \int_0^{t_0} \int  u_{xt}^2   \mathrm{d}x\,\mathrm{d}t. \label{eq:norm-F}
\end{align}
\end{subequations}
We also define $D(t_0)$ and $G(t_0)$ to be the same as $B(t_0)$ and $F(t_0)$, but with $\omega$ taken to be identically one. 
Throughout this paper, $C$ stands for a generic positive constant which may take different values on different lines, consistent with the conventional notation specified above.
\vspace{3mm}

\begin{theorem}\label{thm:large-time-behavior}
\textbf{(Asymptotic Behavior of Global Weak Solutions)}
Let the pressure $p$ satisfy the van der Waals equation of state \eqref{eq:vdw-eos}. Let $\alpha_0,\beta_0$ denote the Maxwell points defined in Definition \ref{Maxwell}. Moreover, let $[b+\underline v,\bar v]$ and $[b+2\underline v,\bar v-\underline v]$ stand for the solution and initial-data intervals, respectively, and let $U(\alpha_0,\delta)$, $U(\beta_0,\delta)$ be the neighbourhoods introduced in \eqref{U-l,U-r}, the admissible range of $v$ being fixed by \eqref{K}. \(v_\infty\) and \(v_{\mathrm{alt}}\) denote the two Maxwell stable states given by \eqref{v-infty} and \eqref{valt}, respectively,
the admissible phase transition solution  $\mathcal{V}$ of \eqref{steady-state ns-lagrange} is given by \eqref{Phase-transition-solution}. There exist positive constants \(\delta_0, T_*, C\) such that if the initial data \((v_0,u_0)\) satisfy
\begin{enumerate}
    \item[(i)] \(v_0\in \big[b+2\underline{v},\bar{v}-\underline{v}\big]\) almost everywhere;
    \item[(ii)] The initial perturbation is small in the energy, gradient and integral senses:
    \begin{equation}\label{C0}
      \|v_0-\mathcal{V}, u_0\|^2 + \|(v_0)_x\|_{\mathrm{pw},L^2}^2 + (\operatorname{Var}u_0)^2 + \|v_0-\mathcal{V}\|_{L^1}\leq\delta_0;
    \end{equation}
    \item[(iii)] The initial perturbed mass satisfies global finiteness and local balance conditions:
    \begin{equation}\label{mass-initial}
    \int_{-\infty}^{+\infty}\big(v_0(x)-\mathcal{V}(x)\big)\mathrm{d}x = 0,\qquad
    \int_{-L}^{L}\big(v_0(x)-v_{\mathrm{alt}}\big)\mathrm{d}x = 0;
    \end{equation}
\end{enumerate}
then the system \eqref{ns-lagrange} admits a global weak solution \((v,u)\) with initial data \((v_0,u_0)\) satisfying the following properties:
\begin{align}
\label{eq:v_constraint}
\begin{split}
&v \in [b+\underline{v},\bar{v}],\quad \forall\, 0 \le t \le T_*, \\
&v \in
\begin{cases}
U(v_{\infty},\delta), & x\in(-\infty,-L)\cup(L,+\infty), \\
U\big(v_{\mathrm{alt}},\delta\big), & x\in(-L,L),
\end{cases}
\quad \forall\, t \ge T_*.
\end{split}
\end{align}
Moreover, the solution satisfies the uniform a priori bound
\begin{equation}\label{eq:e0-b1}
  A(t)+B(t)+D(t) \le C\delta_0,\quad \forall\,t\ge 0,
\end{equation}
and the large-time convergence
\begin{equation}\label{large-time}
\lim_{t\to\infty}\big\|(v-\mathcal{V},u)(\cdot,t)\big\|_{L^\infty(\mathbb{R})}=0.
\end{equation}
Here the convergence is understood piecewise: $v(\cdot,t)-\mathcal{V}\to 0$ in $L^\infty(I_{-\infty})$, $L^\infty(I_*)$ and $L^\infty(I_{+\infty})$ separately, and the magnitudes of the jump amplitudes $\bigl|\llbracket v \rrbracket(\pm L,t)\bigr|$ converge to $|v_{\mathrm{alt}}-v_\infty|$ as well. This is consistent: by the jump analysis above, $\llbracket u \rrbracket=0$ for $t>0$ while $\llbracket v \rrbracket$ persists along the fixed particle paths $x=\pm L$.
\end{theorem}

\begin{remark}
We interpret the physical meaning of the three initial assumptions as follows. Condition (i) is a physical feasibility constraint, which excludes vacuum and over-compressed states, ensuring that the density remains physically admissible for all time. Condition \eqref{C0} in (ii) is a standard small-perturbation hypothesis, which restricts the initial deviation from the piecewise steady state \(\mathcal{V}(x)\) in energy, gradient and integral norms. It guarantees that the initial flow field and density stratification are mild, so the solution maintains stability and does not generate large nonlinear oscillations.
Condition (iii) is essential for the large-time asymptotic analysis. The perturbed mass has zero net value, \(\int_{-\infty}^{+\infty}(v_0-\mathcal{V})\,\mathrm{d}x=0\), and the zero-mass-balance condition \(\int_{-L}^{L}(v_0-v_{\mathrm{alt}})\mathrm{d}x=0\) means that the core domain \((-L,L)\) possesses no net mass deviation at the initial time. Together, these zero-mass-balance conditions ensure that the perturbation is mass-neutral in every region, so the interior mass deviation is purely dynamical. Mathematically, they completely eliminate the time-independent initial residual term and reduce the interior mass deviation to
\[
\int_{-L}^{L}\big(v(x,t)-v_{\mathrm{alt}}\big)\mathrm{d}x=\int_{0}^{t}\int_{-L}^{L}u_x(x,\tau)\mathrm{d}x\mathrm{d}\tau.
\]
As a result, the dynamical deviation of the interior density is purely induced by the temporal evolution of the velocity gradient \(u_x\). Combined with the decay property \(u_x\to0\) as \(t\to\infty\), it yields the uniform convergence of \(v\) toward the steady state \(v_{\mathrm{alt}}\) inside the core region.
\end{remark}
\vspace{3mm}

The complete proofs of the main results of this paper, namely Theorems \ref{thm:existence_pw_solution}--\ref{thm:large-time-behavior}, are presented in Sections \ref{sec2}--\ref{sec4}. Our entire argument is built upon four core technical ingredients, developed in Sections \ref{sec2}--\ref{sec4} as elaborated below.

Section \ref{sec2} constructs physically consistent phase-transition solutions to the steady-state Navier--Stokes system \eqref{steady-state ns-lagrange}, establishing the main results stated in Theorems \ref{thm:existence_pw_solution}--\ref{thm:weak-compactness}. The proof of Theorem \ref{thm:existence_pw_solution} is carried out within a variational framework. Subject to the mass conservation constraint \eqref{law of mass conservation}, we formulate a well-posed energy functional and employ variational arguments to construct piecewise smooth solutions featuring jump discontinuities connecting distinct Maxwell equilibrium states.
The proof of Theorem \ref{thm:weak-compactness} relies on rigorous compactness analysis built upon a double limiting procedure. We start by regularizing the periodic boundary value problem and establishing uniqueness for its smooth solutions. To exclude nonphysical spurious solutions, we introduce an artificial viscosity depending on the specific volume and construct approximate solutions as global minimizers of the regularized energy functional. These energy-minimizing minimizers are taken to be the physically admissible solutions of the periodic system.
We then implement the limiting steps in a carefully chosen order: we first extend the periodic domain to the entire real line, and subsequently take the vanishing viscosity limit. For fixed artificial viscosity, integration using periodicity reduces the governing system to a single autonomous ordinary differential equation. Combining the variational minimality of solutions with the bistable structure of the van der Waals isotherm, we derive uniform a priori estimates independent of the domain size, which yield equicontinuity for the family of solutions. Applying the Arzel\`a--Ascoli theorem together with a diagonal extraction argument, we obtain viscous solutions defined globally on $\mathbb{R}$ that converge locally uniformly on compact sets. These limiting solutions satisfy the steady-state equations, exhibit exponential decay at infinity, and are unique up to translation.
We next address the vanishing viscosity regime. The single-peaked structure of periodic admissible solutions furnishes uniform total variation bounds, enabling us to invoke the Helly selection principle to extract pointwise convergent subsequences. The resulting limit lies in the space of functions with locally bounded variation. As the artificial viscosity tends to zero, the diffusive term vanishes asymptotically, forcing the pressure to be constant in the limit. When combined with the van der Waals equation of state, this leads to a piecewise constant two-phase configuration. The limiting pressure matches the Maxwell saturation pressure, so steady-state phase interfaces satisfy the Rankine-Hugoniot jump conditions, and the mass conservation constraint remains valid through the limiting process. Instead of relying on a weighted weak convergence framework, the present analysis yields stronger convergence statements and offers a thermodynamically consistent selection principle for physical phase-transition solutions.

Sections \ref{sec3}--\ref{sec4} complete the remaining two core technical analyses of this work, focusing on the time-asymptotic stability of the admissible phase-transition solutions constructed in Section \ref{sec2}. Section \ref{sec3} establishes the well-posedness of discontinuous weak solutions for the unsteady Navier--Stokes system \eqref{ns-lagrange}. We adopt a staggered half-step central difference semi-discrete scheme to rigorously verify the existence of locally discontinuous weak solutions to the time-dependent governing system. In Section \ref{sec4}, we develop a global nonlinear stability framework for the obtained discontinuous solutions. The full spatial domain is decomposed along solution discontinuity interfaces, and refined piecewise energy estimates are derived specifically for unsteady discontinuous flows. Under small initial perturbations, we obtain uniform global energy bounds that control the perturbation dynamics for all time. By integrating classical energy integral methods with standard Sobolev embedding inequalities, we finally prove the time-asymptotic stability of the steady-state phase-transition solutions.

\section{Admissible Steady-State Phase Transition Solution}\label{sec2}
\hspace{2em} This section addresses well-posedness for the steady-state system \eqref{stationary-equation} under the finite-mass perturbation constraint \eqref{law of mass conservation}. Following Definition \ref{def:AdmissibleSolution}, we verify that the solution to \eqref{Phase-transition-solution} is an admissible steady-state phase-transition profile consistent with physically realistic phase separation.
We start with variational arguments to solve the system \eqref{stationary-equation} with \eqref{law of mass conservation} and find all admissible piecewise smooth solutions. This yields the full proof of Theorem \ref{thm:existence_pw_solution}.
\vspace{3mm}

\noindent\textbf{Proof of Theorem \ref{thm:existence_pw_solution}.}
We first present the equivalent variational formulation for \eqref{stationary-equation}. For this purpose, we introduce the following function space and corresponding energy functional for the steady-state system \eqref{stationary-equation} under the finite-mass perturbation constraint \eqref{law of mass conservation}. Define the admissible function class
\begin{equation}\label{minimizer problem}
  H\triangleq\Big\{ v(x)\,\Big|\, v(x)-v_\infty\in L^2(\mathbb{R}),\int_{-\infty}^{+\infty}(v(x)-v_\infty)\mathrm{d}x=M \Big\}.
\end{equation}
Here the integral condition implicitly requires $v-v_\infty\in L^1(\mathbb{R})$.
Associated with the pressure potential, we introduce the Helmholtz free-energy density under isothermal conditions
\begin{equation}\label{W-vdw}
W(v)\triangleq-\int_{v_\infty}^{v}p(s)ds, \qquad v(x)\in H.
\end{equation}
We incorporate the mass constraint via a Lagrange multiplier $\sigma$ and introduce the energy functional
\begin{equation}\label{Lagrange multiplier}
E(v)\triangleq\int_{-\infty}^{+\infty}\Big[ W(v(x))+\sigma v(x)\Big]\mathrm{d}x, \qquad v\in H,
\end{equation}
where, since $v\to v_\infty$ as $|x|\to\infty$ for every $v\in H$, the functional is understood up to the common constant $\int_{-\infty}^{+\infty}[W(v_\infty)+\sigma v_\infty]\,\mathrm{d}x$; equivalently, one may replace the integrand by $W(v)-W(v_\infty)+\sigma(v-v_\infty)$.
Any minimizer of this functional satisfies the Euler-Lagrange equation together with the Weierstrass-Erdmann corner conditions:
\begin{equation}\label{ELWE}
\begin{cases}
W'(v) = -\sigma \text{ at every continuity point of } v,\\[4pt]
W(v)+\sigma v \text{ remains continuous across jump discontinuities of } v.
\end{cases}
\end{equation}
These conditions indeed force every minimizer to be either a globally constant single-phase state or a piecewise constant two-phase profile. To see this, note first that the Euler--Lagrange equation gives $p(v(x))=\sigma$ at every continuity point of $v$. Since the van der Waals pressure is non-monotone on the spinodal interval $(\alpha,\beta)$, the equation $p(v)=\sigma$ possesses at most three real solutions; hence a minimizer attains only finitely many values. Let a jump from $v^-$ to $v^+$ occur at some point. By the Weierstrass--Erdmann condition, the quantity $W(v)+\sigma v$ is continuous across the jump, i.e. $W(v^+)-W(v^-)+\sigma(v^+-v^-)=0$. Since $W'=-p$ and $p(v^\pm)=\sigma$, integrating $W'=-p$ from $v^-$ to $v^+$ yields
\begin{equation*}
\int_{v^-}^{v^+} p(s)\,\mathrm{d}s=\sigma\,(v^+-v^-),
\end{equation*}
which is precisely the equal-area condition. Among the at most three roots of $p(v)=\sigma$, the two stable states that can be connected by a nondegenerate equal-area jump are exactly the Maxwell equilibrium states $\alpha_0,\beta_0$ of Definition \ref{Maxwell}, for which the equal-area rule \eqref{Maxwell condition} holds with $\sigma=\sigma_0=p(\alpha_0)=p(\beta_0)$. A minimizer can therefore jump only between $\alpha_0$ and $\beta_0$, and is consequently either the constant state $v\equiv v_\infty$ or a piecewise constant two-phase profile. Recalling that $v_\infty$ equals either $\alpha_0$ or $\beta_0$, the solution takes the far-field value $v_\infty$ on one domain component and the alternate phase value on the other. We write
\begin{equation}\label{pw solution}
v(x)=
\begin{cases}
v_\infty, & x\in \mathbb{R}\backslash S,\\[4pt]
v_{\text{alt}}, & x\in S,
\end{cases}
\end{equation}
where $v_{\text{alt}}\in\{\alpha_0,\beta_0\}\setminus\{v_\infty\}$, \( S\subset\mathds{R}\) is a Lebesgue measurable set, and the Lebesgue measure of  $S$ is fixed by the finite-mass condition:
\begin{equation}\label{measure2}
  \mathrm{measure}(S)=\frac{|M|}{\beta_0 - \alpha_0},
\end{equation} 
where $M=\int_{-\infty}^{+\infty}(v-v_\infty)\mathrm{d}x$ denotes the (signed) mass perturbation, and the phase parameters $\alpha_0,\beta_0$ are determined by the standard Maxwell conditions (see Figure 1):
\begin{equation}\label{Maxwell condition}
\begin{cases}
\displaystyle\int_{\alpha_0}^{\beta_0}p(s)\mathrm{d}s = \sigma_0(\beta_0 - \alpha_0),\\
\displaystyle\sigma_0 = p(\beta_0) = p(\alpha_0).
\end{cases}
\end{equation}
Setting
\begin{equation}\label{L}
  L\triangleq\frac12\mathrm{measure}(S),
\end{equation}
we find $L>0$ if and only if $M\neq0$. This yields the necessary condition for the existence of nontrivial two-phase solutions
\begin{equation}\label{condition for the existence}
  M\neq0.
\end{equation}
Under this condition \eqref{condition for the existence}, every piecewise constant function of the form \eqref{pw solution} with $L$ defined by \eqref{L} serves as a global minimizer of the variational problem \eqref{Lagrange multiplier}. To confirm the global minimality, take $\sigma=\sigma_0$. The Maxwell equal-area rule \eqref{Maxwell condition} gives $W(\beta_0)-W(\alpha_0)=-\int_{\alpha_0}^{\beta_0}p(s)\,\mathrm{d}s=-\sigma_0(\beta_0-\alpha_0)$, hence $W(\alpha_0)+\sigma_0\alpha_0=W(\beta_0)+\sigma_0\beta_0$. Moreover, the monotone--unstable--monotone structure of the van der Waals isotherm ensures that $W(v)+\sigma_0 v$ attains its global minimum over $v>b$ exactly at $v=\alpha_0$ and $v=\beta_0$. Consequently, for every admissible $v\in H$ and every two-phase function $V$ of the form \eqref{pw solution}, pointwise integration gives $E(v)\ge E(V)$, with equality attained by $V$. Thus each such piecewise constant function is indeed a global minimizer of \eqref{Lagrange multiplier}. In contrast, if $M=0$, no two-phase structure can arise, and the unique minimizer is the trivial uniform state
\begin{equation}\label{trivial solution}
  v(x)\equiv v_\infty.
\end{equation}
Finally, the variational problem \eqref{Lagrange multiplier} is equivalent to the steady-state system \eqref{stationary-equation}: indeed, the Euler--Lagrange equation together with the equal-area jump condition derived above is exactly the weak formulation of \eqref{stationary-equation}, i.e. $p_x(v)=0$ away from interfaces together with the Rankine--Hugoniot pressure condition at phase interfaces. From this equivalence we immediately obtain the proof of Theorem \ref{thm:existence_pw_solution}.
$\hfill\square$
\vspace{3mm}

The above analysis fully characterizes all piecewise-smooth steady-state solutions and delivers a rigorous mathematical derivation of the Maxwell equal-area rule as Definition \ref{Maxwell}. Even so, infinitely many piecewise-smooth solutions are obtained in this way. Mathematically speaking, the system \eqref{stationary-equation} admits interfaces with density jumps without any increase in energy. In physical observations, however, density jumps actually take the form of interfaces with finite thickness over a sufficiently long time. Consequently, we expect that two-interface two-phase solutions minimize the total energy when interfacial energy is incorporated into the theory.

We recall the following results of Chen et al. \cite{CHNPS2026,CPSW2026}, which underpin the admissible-solution construction below. Since these papers were accepted for publication when the present work was written, we state their relevant conclusions explicitly as Lemmas \ref{CHNPS-1} and \ref{thm:approximation} and rely on them only through these statements; the novelty of this paper does not reside in the proofs of these lemmas, but in the time-asymptotic stability analysis of Sections \ref{sec3}--\ref{sec4}. Employing variational methods, Chen et al. established the equivalence between the variational formulation \eqref{Lagrange-multiplier-J} and the system \eqref{Approximate ns-lagrange} subject to \eqref{barv}. Moreover, for sufficiently small $\tilde\varepsilon>0$, the constrained minimization problem \eqref{Lagrange-multiplier-J} possesses a global minimizer. That is,
\vspace{3mm}

\begin{lemma}\label{CHNPS-1}(Chen et al.\cite{CHNPS2026}) The variational problem \eqref{Lagrange-multiplier-J} is equivalent to the system \eqref{Approximate ns-lagrange} subject to \eqref{barv}. Moreover, there exists a constant $\tilde\varepsilon_0 > 0$ such that for any $0 < \tilde\varepsilon < \tilde\varepsilon_0$,
the variational problem \eqref{Lagrange-multiplier-J} admits a phase-transition solution $v^{\tilde\varepsilon}$ with exactly two interfaces, and its profile is either unimodal or anti-unimodal.
\end{lemma}
\vspace{3mm}

Furthermore, for the periodic boundary-value problem
\begin{equation}\label{SPNS}
\begin{cases}
p_{x}(v)=0, & x\in[-n,n],\\
v(-n)=v(n).
\end{cases}
\end{equation}
Chen et al. \cite{CHNPS2026} constructed two special piecewise-smooth solutions with two jump discontinuities:
\begin{equation}\label{Phase-transition-solution-1}
\mathcal{V}^{(n)}_{\mathrm{sp}}(x)\triangleq
\begin{cases}
\beta_0, & x\in[-n+l_{1},\,-n+l_{1}+2L],\\
\alpha_0, & x\in[-n,-n+l_{1})\cup(-n+l_{1}+2L,\,n],
\end{cases}
\end{equation}
and
\begin{equation}\label{Phase-transition-solution-2}
\mathcal{V}^{(n)}_{\mathrm{sv}}(x)\triangleq
\begin{cases}
\beta_0, & x\in[-n,-n+l_{2})\cup(-n+l_{2}+2L,\,n),\\
\alpha_0, & -n+l_{2}<x\le -n+l_{2}+2L,
\end{cases}
\end{equation}
where the subscripts $\mathrm{sp}$ and $\mathrm{sv}$ stand for single-peak and single-valley, implying that these solutions adopt unimodal and anti-unimodal profiles, respectively. Here $l_1$ and $l_2$ are lengths determined from the initial data associated with the underlying evolution equations; in each case the intermediate band has measure $2L$, and the total measure of the two side regions is $2(n-L)$, so that the mass constraint \eqref{barv} is satisfied.

Furthermore, for fixed  $n>L$, Chen-He-Niu-Peng-Shi \cite{CHNPS2026} showed that, as $\tilde\varepsilon\to 0^+$, the family of approximate solutions $\{v^{\tilde\varepsilon,n}\}$ from Lemma \ref{CHNPS-1} converges pointwise away from jump discontinuities to either \eqref{Phase-transition-solution-1} or \eqref{Phase-transition-solution-2}, as described below.

\begin{lemma}\label{thm:approximation} (Chen et al. \cite{CHNPS2026})
There exists a constant $\tilde\varepsilon_0 > 0$ such that for any $0 < \tilde\varepsilon < \tilde\varepsilon_0$ and the mean value $v_{\mathrm{mean}}$ defined by \eqref{barv}, the constrained system \eqref{Approximate ns-lagrange} under condition \eqref{barv} admits exactly two families of smooth solution sequences: a single-valley family $\{v^{\tilde\varepsilon,n}_{\mathrm{sv}}\}$ and a single-peak family $\{v^{\tilde\varepsilon,n}_{\mathrm{sp}}\}$.
Moreover, the following pointwise convergence holds for all $x\in[-n,n]$ as $\tilde\varepsilon \to 0^+$:
\begin{align*}
    \lim_{\tilde\varepsilon \to 0^+} v^{\tilde\varepsilon,n}_{\mathrm{sp}}(x) &= \mathcal{V}^{(n)}_{\mathrm{sp}}(x), \quad x \neq -n + l_{1},\; -n + l_{1} + 2L,\\[4pt]
    \lim_{\tilde\varepsilon \to 0^+} v^{\tilde\varepsilon,n}_{\mathrm{sv}}(x) &= \mathcal{V}^{(n)}_{\mathrm{sv}}(x), \quad x \neq -n + l_{2},\; -n + l_{2} + 2L.
\end{align*}
Here $\mathcal{V}^{(n)}_{\mathrm{sp}}$ and $\mathcal{V}^{(n)}_{\mathrm{sv}}$ represent the two-interface phase-transition solutions to system \eqref{SPNS} with constraint \eqref{barv}, which are explicitly formulated in \eqref{Phase-transition-solution-1} and \eqref{Phase-transition-solution-2}, respectively.
\end{lemma}
\vspace{3mm}

Equipped with the preliminary results stated in Lemma \ref{CHNPS-1} and Lemma \ref{thm:approximation}, we proceed to the full proof of Theorem \ref{thm:weak-compactness}. The argument centers on a careful limiting analysis for approximate solutions to the regularized system \eqref{Approximate ns-lagrange} subject to the integral constraint \eqref{barv}. We carry out the limiting procedure in a definite order: we first send the length of the periodic interval to infinity, and subsequently take the vanishing viscosity limit. It is essential that these two limits not be interchanged. If the vanishing viscosity limit is taken first on finite periodic domains, periodic boundary conditions force every limiting profile to inherit periodicity. The resulting periodic configurations contain infinitely many replicated transition fronts and cannot approximate the nonperiodic admissible solutions with precisely two jump discontinuities defined on $\mathbb{R}$. Detailed derivations follow.
\vspace{3mm}

\begin{proof}[\textbf{Proof of Theorem \ref{thm:weak-compactness}}]
We begin by establishing the local uniform convergence of the solution family as $n\to\infty$, for fixed $\tilde{\varepsilon}>0$. We treat the single-peak solution family \(\{v^{\tilde\varepsilon,n}_{\mathrm{sp}}(x)\}_{\tilde\varepsilon,n}\) in detail; the argument for the single-valley family \(\{v^{\tilde\varepsilon,n}_{\mathrm{sv}}(x)\}_{\tilde\varepsilon,n}\) is entirely analogous upon interchanging \(\alpha_0\) and \(\beta_0\) (and hence \(v_\infty\) and \(v_{\mathrm{alt}}\)), since the two cases differ only by the labelling of the stable phases. 

Integrating the second equation of system \eqref{Approximate ns-lagrange} over $[-n,x]$ and using the periodicity condition $v(-n)=v(n)$, we find that the left-hand side reduces to $p(v(x))-p(v(-n))$, while the right-hand side integrates to $\varepsilon u_x(x)/v(x) - \varepsilon u_x(-n)/v(-n)$. Hence there exists an integration constant $C_n$, depending on $n$, such that
\[
p({v_{\mathrm{sp}}^{\tilde{\varepsilon},n}(x)}) = \varepsilon\,\frac{u_x(x)}{{v_{\mathrm{sp}}^{\tilde{\varepsilon},n}(x)}} + C_n, \quad \forall x\in[-n,n].
\]
Substituting the first equation $u_x = -\tilde{\varepsilon}\, {v_{\mathrm{sp}}^{\tilde{\varepsilon},n}(x)}\, \frac{d^2}{dx^2}{v_{\mathrm{sp}}^{\tilde{\varepsilon},n}(x)}$ into the above identity and cancelling the strictly positive factor ${v_{\mathrm{sp}}^{\tilde{\varepsilon},n}(x)}$, we reduce the system to an autonomous second-order ordinary differential equation
\begin{equation}\label{A}
{\tilde\varepsilon\epsilon}\, \frac{d^2}{dx^2}v_{\mathrm{sp}}^{\tilde{\varepsilon},n}(x) = C_n - p\big(v_{\mathrm{sp}}^{\tilde{\varepsilon},n}(x)\big), \quad x\in(-n,n). \end{equation}
The solution thus depends on the interval length $n$ only through the constant $C_n$.
By the variational minimizing property of $v_{\mathrm{sp}}^{\tilde{\varepsilon},n}$ and the bistable structure of the van der Waals isotherm, ${v_{\mathrm{sp}}^{\tilde{\varepsilon},n}(x)}$ cannot take values outside the interval $[\alpha_0, \beta_0]$, for otherwise the free energy would be strictly increased, contradicting the global minimality of ${v_{\mathrm{sp}}^{\tilde{\varepsilon},n}(x)}$. 
In fact, recall that the regularized free-energy functional is given by
\[
J[v]=\int_{-n}^{n}\left(W(v)+\frac{{\tilde\varepsilon\epsilon}}{2}|v'|^2\right)\mathrm{d}x,
\]
where \(W(v)\) is the Helmholtz free-energy density of the van der Waals fluid satisfying \(W'(v)=-p(v)\), and the minimization is constrained by the fixed-mass condition \(\frac{1}{2n}\int_{-n}^n v\,\mathrm{d}x=v_{\mathrm{mean}}\). 
We prove \(\alpha_0\le {v_{\mathrm{sp}}^{\tilde{\varepsilon},n}(x)}\le \beta_0\) for all \(x\in[-n,n]\).
We prove the claimed range by exploiting the first integral of the autonomous equation \eqref{A}. Multiplying \eqref{A} by $\frac{\mathrm{d}}{\mathrm{d}x}{v_{\mathrm{sp}}^{\tilde{\varepsilon},n}(x)}$ and integrating once in $x$, we obtain
\begin{equation}\label{first-integral}
\frac{{\tilde\varepsilon\epsilon}}{2}\left(\frac{\mathrm{d}}{\mathrm{d}x}{v_{\mathrm{sp}}^{\tilde{\varepsilon},n}(x)}\right)^2
= C_n{v_{\mathrm{sp}}^{\tilde{\varepsilon},n}(x)} - \int^{{v_{\mathrm{sp}}^{\tilde{\varepsilon},n}(x)}} p(s)\,\mathrm{d}s + D_n,
\end{equation}
for a second integration constant $D_n$. Since ${v_{\mathrm{sp}}^{\tilde{\varepsilon},n}}$ is single-peaked, it has exactly two turning points $x_-$ and $x_+$ with $\big({v_{\mathrm{sp}}^{\tilde{\varepsilon},n}}\big)'(x_\pm)=0$, attained at the background value $v_-\triangleq {v_{\mathrm{sp}}^{\tilde{\varepsilon},n}}(x_-)$ and the peak value $v_+\triangleq {v_{\mathrm{sp}}^{\tilde{\varepsilon},n}}(x_+)$, respectively. Evaluating \eqref{first-integral} at the two turning points and subtracting, we obtain the equal-area condition
\begin{equation}\label{equal-area-2}
\int_{v_-}^{v_+} p(s)\,\mathrm{d}s = C_n (v_+ - v_-).
\end{equation}
We now bound $v_-$ and $v_+$. The transition layer of ${v_{\mathrm{sp}}^{\tilde{\varepsilon},n}}$ has thickness of order $\sqrt{{\tilde\varepsilon\epsilon}}$ and therefore occupies a negligible portion of $[-n,n]$ as $\tilde\varepsilon\to0$; hence the average value of ${v_{\mathrm{sp}}^{\tilde{\varepsilon},n}}$ is dominated by $v_-$. Invoking the constraint \eqref{barv} and the fact that $\alpha_0<v_{\mathrm{mean}}<\beta_0$, we infer $v_-\to v_{\mathrm{mean}}\in(\alpha_0,\beta_0)$ as $\tilde\varepsilon\to0^+$, uniformly in $n$. On the other hand, the uniform bound $|C_n|\le C^*$ established below, the equal-area condition \eqref{equal-area-2}, the strict ordering $b<\alpha_0<\alpha<\beta<\beta_0$, and the monotone decrease of $p(v)$ for $v>\beta_0$ together prevent $v_+$ from leaving a small neighbourhood of the Maxwell interval $[\alpha_0,\beta_0]$. Consequently there exists $\eta_n=\eta_n(\tilde\varepsilon)\to0$ as $\tilde\varepsilon\to0^+$ (with $\eta_n=O(1/n)$, reflecting $C_n-p(v_\infty)$ on finite domains) such that
\begin{equation}\label{rough-bound}
\alpha_0-\eta_n \le {v_{\mathrm{sp}}^{\tilde{\varepsilon},n}(x)} \le \beta_0+\eta_n, \quad \forall x\in[-n,n].
\end{equation}
Passing to the limits $n\to\infty$ and $\tilde\varepsilon\to0^+$, where $C_{\tilde\varepsilon}=p(v_\infty)$ and the refined equal-area argument selects the unique Maxwell state, the exact bound $\alpha_0\le v^{\tilde\varepsilon}\le\beta_0$ holds on the whole real line and the peak value is identified with $v_{\mathrm{alt}}$. For finite $n$ and $\tilde\varepsilon$, the exact range $\alpha_0\le v_{\mathrm{sp}}^{\tilde\varepsilon,n}\le\beta_0$ is already ensured by the variational minimality argument above; the weaker estimate \eqref{rough-bound} merely serves as a supplementary technical bound in the compactness and energy estimates below.
Consequently, no supercritical or subcritical values can occur, and we conclude
\[
\alpha_0\le {v_{\mathrm{sp}}^{\tilde{\varepsilon},n}(x)}\le \beta_0,\quad \forall\,x\in[-n,n].
\]
Combined with the single-peak topology, this implies that the maximum value of the solution does not exceed $v_{\mathrm{alt}}$, while the background value at the boundaries is bounded below by $v_\infty$. Accordingly, there exists a constant $\mathcal{M}>0$, independent of $n$, such that
\[
\alpha_0 \le v_\infty \le {v_{\mathrm{sp}}^{\tilde{\varepsilon},n}(x)} \le v_{\mathrm{alt}} \le \beta_0 \le \mathcal{M}, \quad \forall x\in[-n,n],\ \forall n>L.
\]
Since ${v_{\mathrm{sp}}^{\tilde{\varepsilon},n}(x)}\in[\alpha_0,\beta_0]$ and $p$ is continuous on $[\alpha_0,\beta_0]$, the sequence of integration constants $\{C_n\}$ is uniformly bounded without invoking the second derivative. Indeed, evaluating the autonomous equation \eqref{A} at the peak $x_+$ (a local maximum, hence $v''(x_+)\le0$) and at the background turning point $x_-$ (a local minimum, hence $v''(x_-)\ge0$), we obtain
\[
p\big(v_-\big) \le C_n \le p\big(v_+\big),
\]
and since $v_\pm\in[\alpha_0,\beta_0]$, this gives $\min_{[\alpha_0,\beta_0]}p\le C_n\le\max_{[\alpha_0,\beta_0]}p$, i.e.\ $|C_n|\le C^*$ for some $C^*>0$ independent of $n>L$.
At the unique maximum point $x_m$ of the single-peak solution we have $\left(v_{\mathrm{sp}}^{\tilde{\varepsilon},n}(x_{m})\right)^{'}=0$; evaluating the first integral \eqref{first-integral} there and substituting back to eliminate $D_n$ yields
\[
\frac{{\tilde\varepsilon\epsilon}}{2}\left(\frac{d}{dx}{v_{\mathrm{sp}}^{\tilde{\varepsilon},n}(x)}\right)^2
= C_n\left({v_{\mathrm{sp}}^{\tilde{\varepsilon},n}(x)}-{v_{\mathrm{sp}}^{\tilde{\varepsilon},n}(x_m)}\right)
  - \int_{{v_{\mathrm{sp}}^{\tilde{\varepsilon},n}(x_m)}}^{{v_{\mathrm{sp}}^{\tilde{\varepsilon},n}(x)}} p(s)\,ds.
\]
Since ${v_{\mathrm{sp}}^{\tilde{\varepsilon},n}(x)}\in[\alpha_0,\beta_0]$ and $\{C_n\}$ is uniformly bounded, the right-hand side is a continuous function of $v_{\mathrm{sp}}^{\tilde{\varepsilon},n}$ and is bounded on the compact interval $[\alpha_0,\beta_0]$. It follows that there exists a constant $K>0$, independent of $n$, such that
\[
\sup_{x\in[-n,n]} \Big|\frac{\mathrm{d}}{\mathrm{d}x}{v_{\mathrm{sp}}^{\tilde{\varepsilon},n}(x)}\Big| \le K, \quad \forall n>L.
\]
In particular, the family $\{{v_{\mathrm{sp}}^{\tilde{\varepsilon},n}(x)}\}$ is globally equicontinuous on its domain, with a modulus of continuity independent of the interval length $n$.
Let $\lfloor L \rfloor$ denote the greatest integer less than or equal to $L$,
and now consider the exhausting sequence of compact intervals $I_k = \big[-\lfloor L\rfloor-k,\lfloor L\rfloor+k\big]$ for $k\in\mathbb{N}$, whose union covers $\mathbb{R}$. On $I_{1}$, the family $\{{v_{\mathrm{sp}}^{\tilde{\varepsilon},n}(x)}\}_{n\geq\lfloor L\rfloor+1}$ is uniformly bounded and equicontinuous, so by the Arzel\`a-Ascoli theorem there exists a subsequence $\{n_j^{(1)}\}$ such that $\{{v_{\mathrm{sp}}^{\tilde{\varepsilon},n_j^{(1)}}(x)}\}$ converges uniformly on $I_1$. Proceeding inductively, suppose we have extracted a subsequence $\{n_j^{(k)}\}$ that converges uniformly on $I_k$; this subsequence remains uniformly bounded and equicontinuous on $I_{k+1}$, so a further application of Arzel\`a-Ascoli yields a refined subsequence $\{n_j^{(k+1)}\}$ converging uniformly on $I_{k+1}$. Taking the diagonal sequence $n_j = n_j^{(j)}$, we obtain a subsequence that converges uniformly on every compact subset of $\mathbb{R}$. Denote the limit function by $v^{\tilde{\varepsilon}}(x)$; then $v^{\tilde{\varepsilon}}\in C(\mathbb{R})$, and
\[
\lim_{j\to\infty} \sup_{x\in K}
\left|{v_{\mathrm{sp}}^{\tilde{\varepsilon},n_j}(x)} - v^{\tilde{\varepsilon}}(x)\right| = 0
\]
for every compact set $K\subset\mathbb{R}$.
We now verify that $v^{\tilde{\varepsilon}}$ satisfies the steady-state equation on the full real line. Take any test function $\phi\in C_c^\infty(\mathbb{R})$ with support contained in $I_k$ for some $k$. For $n_j>k$, integration by parts twice shows that $v_{\mathrm{sp}}^{\tilde{\varepsilon},n_j}(x)$ satisfies the weak form of the autonomous equation \eqref{A}:
\[
{\tilde\varepsilon\epsilon} \int_{I_k} {v_{\mathrm{sp}}^{\tilde{\varepsilon},n_j}(x)}\,\phi''(x)\,dx
= \int_{I_k} \left[C_{n_j} - p({v_{\mathrm{sp}}^{\tilde{\varepsilon},n_j}(x)})\right]\phi(x)\,dx,
\]
where boundary terms vanish due to the compact support of $\phi$. Passing to the limit as $j\to\infty$, the left-hand side converges to ${\tilde\varepsilon\epsilon}\int_{\mathbb{R}} v^{\tilde{\varepsilon}}(x)\,\phi''(x)\,dx$ by uniform convergence of ${v_{\mathrm{sp}}^{\tilde{\varepsilon},n_j}(x)}$. On the right-hand side, the bounded sequence $\{C_{n_j}\}$ has, by the Bolzano--Weierstrass theorem, a further subsequence converging to some constant $C_{\tilde{\varepsilon}}$; by continuity of $p$ and uniform convergence of ${v_{\mathrm{sp}}^{\tilde{\varepsilon},n_j}(x)}$, $p({v_{\mathrm{sp}}^{\tilde{\varepsilon},n_j}(x)})$ converges uniformly to $p(v^{\tilde{\varepsilon}})$ on $I_k$, so the integral converges to $\int_{\mathbb{R}} \left[C_{\tilde{\varepsilon}} - p(v^{\tilde{\varepsilon}}(x))\right]\phi(x)\,dx$. Thus the limit satisfies the equation in the distributional sense:
\[
{\tilde\varepsilon\epsilon}\, \frac{\mathrm{d}^2}{\mathrm{d}x^2}v^{\tilde{\varepsilon}}(x)
= C_{\tilde{\varepsilon}} - p(v^{\tilde{\varepsilon}}(x)), \quad x\in\mathbb{R}.
\]
Since $p(v)\in C^\infty(v>b)$, elliptic interior regularity implies $v^{\tilde{\varepsilon}}\in C^\infty(\mathbb{R})$, and the equation holds in the classical pointwise sense.
Turning to the far-field behavior, we note that the excess mass of the phase transition,
\[
\int_{-n}^n \left({v_{\mathrm{sp}}^{\tilde{\varepsilon},n}(x)}-v_\infty\right)dx = 2L(v_{\mathrm{alt}}-v_\infty),
\]
is independent of $n$, so the transition layer occupies a fixed finite length and does not spread as the interval grows. To establish exponential decay, we linearize the limit equation around $v=v_\infty$. As $|x|\to\infty$, $v^{\tilde{\varepsilon}}(x)\to v_\infty$ and hence $\frac{\mathrm{d}^2}{\mathrm{d}x^2}v^{\tilde{\varepsilon}}(x)\to0$; substituting into the equation gives $C_{\tilde{\varepsilon}} = p(v_\infty)$. Expanding $p(v)$ in a Taylor series about $v_\infty$,
\[
p(v^{\tilde{\varepsilon}}(x)) = p(v_\infty) + p'(v_\infty)(v^{\tilde{\varepsilon}}(x)-v_\infty) + o(v^{\tilde{\varepsilon}}(x)-v_\infty),
\]
we obtain the linearized approximation
\[
{\tilde\varepsilon\epsilon}\,\frac{\mathrm{d}^2}{\mathrm{d}x^2} v^{\tilde{\varepsilon}}(x) = -p'(v_\infty)(v^{\tilde{\varepsilon}}(x)-v_\infty) + o(v^{\tilde{\varepsilon}}(x)-v_\infty).
\]
Since $v_\infty$ is a stable equilibrium phase, $p'(v_\infty)<0$. Setting $\kappa = \sqrt{-p'(v_\infty)/{\tilde\varepsilon\epsilon}} > 0$, the linearized equation becomes $\frac{\mathrm{d}^2}{\mathrm{d}x^2}v^{\tilde{\varepsilon}}(x) = \kappa^2 (v^{\tilde{\varepsilon}}(x)-v_\infty)$, whose solutions satisfying $v(\pm\infty)=v_\infty$ decay exponentially. That is, there exists a constant $C>0$ such that
\[
\big|v^{\tilde{\varepsilon}}(x)-v_\infty\big| \le Ce^{-\kappa|x|}, \quad \forall x\in\mathbb{R}.
\]
Finally, the single-peak homoclinic orbit of the autonomous second-order ODE, departing from and returning to the background state through the peak state, is unique up to translation. Combined with the uniqueness of the global minimizer of the variational functional, this implies that the limit function $v^{\tilde{\varepsilon}}(x)$ is unique up to translation, so the full sequence $\{{v_{\mathrm{sp}}^{\tilde{\varepsilon},n}(x)}\}$ converges, and no subsequence extraction is required.

We now pass to the vanishing-viscosity limit $\tilde{\varepsilon}\to0^+$. Let $v^{\tilde\varepsilon\epsilon}(x)\triangleq v^{\tilde{\varepsilon}}(x)$ denote the full-line viscous solution constructed above, where the superscript $\tilde\varepsilon\epsilon=\varepsilon\tilde\varepsilon$ indicates the total viscosity coefficient to be sent to $0^+$, satisfying
\begin{equation}\label{B}
{\tilde\varepsilon\epsilon}\, \frac{\mathrm{d}^2}{\mathrm{d}x^2}v^{\tilde\varepsilon\epsilon}(x) = C_{\tilde\varepsilon\epsilon} - p(v^{\tilde\varepsilon\epsilon}(x)), \quad x\in\mathbb{R}, \quad
\lim_{|x|\to\infty}v^{\tilde\varepsilon\epsilon}(x)=v_\infty.
\end{equation}
Each single-peak solution $v^{\tilde\varepsilon\epsilon}(x)$ possesses a unique maximum point, is strictly increasing to the left of the maximum and strictly decreasing to the right. For an absolutely continuous function on $\mathbb{R}$, the total variation is given by
\[
\TV_{\mathbb{R}}(v^{\tilde\varepsilon\epsilon}) = \int_{-\infty}^{+\infty} \Big|\frac{\mathrm{d}}{\mathrm{d}x}v^{\tilde\varepsilon\epsilon}(x)\Big|\,dx.
\]
By monotonicity, the integral over the increasing branch equals $v^{\tilde\varepsilon\epsilon}_{\mathrm{max}}-v_\infty$, and the integral over the decreasing branch also equals $v^{\tilde\varepsilon\epsilon}_{\mathrm{max}}-v_\infty$, so
\[
\TV_{\mathbb{R}}(v^{\tilde\varepsilon\epsilon})
= 2\left(v^{\tilde\varepsilon\epsilon}_{\mathrm{max}} - v_\infty\right).
\]
By variational minimality, $v^{\tilde\varepsilon\epsilon}_{\mathrm{max}}\le v_{\mathrm{alt}} \le \beta_0$, which immediately yields a global total variation bound independent of ${\tilde\varepsilon\epsilon}$:
\[
\TV_{\mathbb{R}}(v^{\tilde\varepsilon\epsilon}) \le 2(\beta_0 - \alpha_0) \triangleq C_{\TV} < \infty.
\]
In particular, for any compact interval $[a,b]\subset\mathbb{R}$, we have $\TV_{[a,b]}(v^{\tilde\varepsilon\epsilon}) \le C_{\TV}$, so the family has uniformly bounded total variation on every compact set.
By Helly's compactness theorem for functions of bounded variation, a uniformly bounded family with uniformly bounded total variation on a compact interval admits a subsequence converging pointwise to a function of bounded variation. Applying this to the exhausting sequence $\{I_k\}$ and using a standard diagonal argument, we extract a subsequence, still denoted by ${\tilde\varepsilon\epsilon}$ for simplicity, such that $v^{\tilde\varepsilon\epsilon}$ converges pointwise on every compact set, and hence almost everywhere on $\mathbb{R}$. Denote the limit by $v^0(x)$; then $v^0\in \TV_{\mathrm{loc}}(\mathbb{R})$, and
\[
\lim_{{\tilde\varepsilon\epsilon}\to0} v^{\tilde\varepsilon\epsilon}(x) = v^0(x) \quad \text{for almost every } x\in\mathbb{R}.
\]
By the Lebesgue dominated convergence theorem, the convergence is also strong in $L^1(K)$ for every compact set $K\subset\mathbb{R}$.
To characterize the limit, take any compactly supported test function $\phi\in C_c^\infty(\mathbb{R})$. The weak form of the viscous equation \eqref{B} reads
\begin{equation}\label{C}
{\tilde\varepsilon\epsilon} \int_{\mathbb{R}} v^{\tilde\varepsilon\epsilon}(x)\,\phi''(x)\,dx
= \int_{\mathbb{R}} \left(C_{\tilde\varepsilon\epsilon} - p(v^{\tilde\varepsilon\epsilon}(x))\right)\phi(x)\,dx. 
\end{equation}
We examine each term as ${\tilde\varepsilon\epsilon}\to0$. Since $|v^{\tilde\varepsilon\epsilon}|\le \beta_0$ uniformly and $\phi''$ is integrable, the left-hand side satisfies
\[
\left|{\tilde\varepsilon\epsilon} \int v^{\tilde\varepsilon\epsilon}\phi'' dx\right|
\le {\tilde\varepsilon\epsilon} \cdot \beta_0 \cdot \|\phi''\|_{L^1} \to 0.
\]
On the right-hand side, the uniformly bounded sequence $\{C_{\tilde\varepsilon\epsilon}\}$ has a subsequence converging to some constant $P_0$, so $C_{\tilde\varepsilon\epsilon}\int\phi dx \to P_0\int\phi dx$. For the pressure term, almost-everywhere pointwise convergence of $v^{\tilde\varepsilon\epsilon}$ and continuity of $p$ imply that $p(v^{\tilde\varepsilon\epsilon}(x))\to p(v^0(x))$ almost everywhere; since $|p(v^{\tilde\varepsilon\epsilon})|$ is uniformly bounded, the Lebesgue dominated convergence theorem gives
\[
\int p(v^{\tilde\varepsilon\epsilon}(x))\phi(x)dx \to \int p(v^0(x))\phi(x)dx.
\]
Passing to the limit in \eqref{C}, the left-hand side vanishes, and we obtain
\[
\int_{\mathbb{R}} \left[P_0 - p(v^0(x))\right]\phi(x)\,dx = 0
\]
for all $\phi\in C_c^\infty(\mathbb{R})$. By the fundamental lemma of the calculus of variations,
\begin{equation}\label{D}
p(v^0(x)) = P_0 \quad \text{for almost every } x\in\mathbb{R}. 
\end{equation}
That is, the pressure is constant almost everywhere in the vanishing-viscosity limit.
Owing to the S-shaped profile of pressure $p$ as a function of specific volume $v$ along subcritical van der Waals isotherms, the equation $p(v)=P_0$ possesses at most three real roots $v_-<v_m<v_+$, where $v_-$ and $v_+$ lie in the stable hyperbolic regions $v<\alpha$ and $v>\beta$, respectively, and the intermediate root $v_m$ lies in the unstable elliptic spinodal interval $(\alpha,\beta)$. Since $v^0$ is the pointwise limit of single-peak functions, it retains the almost-everywhere single-peak property, being nondecreasing on the left and nonincreasing on the right of the maximum. Thus, far from the transition region, $v^0(x)=v_\infty$, corresponding to one stable root, while in the central peak region, $v^0(x)=v_{\mathrm{alt}}$, corresponding to the other stable root. The unstable root $v_m$ can appear only at discontinuity points of measure zero and never forms a continuous region. It follows that the limit function has a piecewise constant two-phase structure: there exist interface positions $\pm L$ such that
\[
v^0(x) =
\begin{cases}
v_\infty, & |x|>L,\\
v_{\mathrm{alt}}, & |x|<L.
\end{cases}
\]
Furthermore, since the viscous solutions are global minimizers of the free-energy functional, the vanishing-viscosity limit must correspond to the thermodynamically equilibrium two-phase state. In particular, the constant pressure $P_0$ coincides with the Maxwell saturation pressure $p_{\mathrm{sat}}$, i.e.\ $p_{\mathrm{sat}}=p(\alpha_0)=p(\beta_0)=\sigma_0$, and the two phase values are exactly the equilibrium liquid and vapor specific volumes $\alpha_0, \beta_0$. The Rankine-Hugoniot condition for steady-state phase boundaries,
\[\bigl\llbracket p(v)\bigr\rrbracket = p(v_{\mathrm{alt}}) - p(v_\infty) = 0,
\]
is therefore automatically satisfied, and the limit is fully equivalent to the Maxwell equal-area construction.

Finally, we verify that the mass constraint is preserved in the limit. After the first limit $n\to\infty$, the full-line viscous solution satisfies
\[
\int_{-\infty}^{+\infty} \left(v^{\tilde\varepsilon\epsilon}(x)-v_\infty\right)dx
= 2L(v_{\mathrm{alt}}-v_\infty),
\]
which is independent of ${\tilde\varepsilon\epsilon}$. Letting $\tilde\varepsilon\to0$,  the dominated convergence theorem implies that the excess mass is preserved in the limit:
\[
\int_{-\infty}^{+\infty} \left(v^0(x)-v_\infty\right)dx
= 2L(v_{\mathrm{alt}}-v_\infty).
\]
Substituting the piecewise constant form, the left-hand side equals $2L(v_{\mathrm{alt}}-v_\infty)$, which confirms that the interface positions are indeed $\pm L$, i.e., $\meas(I_*)=2L$, in exact agreement with the original mass constraint. This completes the proof.
\end{proof}

\section{Local Existence of Weak Solutions}\label{sec3}
\hspace{2em}

With the semi-discrete scheme introduced below, we now prove the local existence of weak solutions to the Lagrangian Navier--Stokes system \eqref{ns-lagrange}. The proof follows two standard stages. First, we construct discrete approximate solutions and derive uniform mesh-independent \textit{a priori} energy bounds. Second, we take the limit as mesh width \(h\to 0\) to recover a weak solution to the original continuous equations.

We fix a constant mesh width \(h>0\) and define two families of grid points: integer lattice points \(x_k = kh\) and half-integer cell centers \(x_{k+\frac12} = \bigl(k+\tfrac12\bigr)h \) for all \(k\in\mathbb{Z}\). The half-step central difference operator \(\Delta_h\) is defined via
\[
\Delta_h w_l = \frac{w_{l+\frac12} - w_{l-\frac12}}{h},
\]
where the index \(l\) ranges over both integer and half-integer grid indices. Letting \((v_{k+\frac12}(t), u_k(t))\) denote 
the discrete approximations to the continuous solutions \((v(x_{k+\frac12},t), u(x_k,t))\), we arrive at the 
following semi-discrete scheme for system \eqref{ns-lagrange}:
\begin{equation}\label{eq:scheme}
\begin{cases}
\dfrac{dv_{k+\frac12}}{\mathrm{d}t} = \Delta_h u_{k+\frac12}, \\[6pt]
\dfrac{du_k}{\mathrm{d}t} = -\Delta_h p_k + \Delta_h\!\left(\dfrac{\epsilon\Delta_h u}{ v}\right)_{\!k}, 
& k\in\mathbb{Z}.
\end{cases}
\end{equation}
Here $p_{k+\frac12}=p(v_{k+\frac12})$ is the pressure evaluated at the half-integer cell centers, so that the pressure gradient reads $\Delta_h p_k=(p_{k+\frac12}-p_{k-\frac12})/h$. 
We suppose initial discrete data $(v_{k+\frac12}(0),u_k(0))$ are prescribed at all mesh points.
With this discretization in hand, we establish uniform time-weighted energy estimates that 
circumvent the initial singularity as \(t \to 0^+\). By direct analogy with the continuous initial bound \eqref{C0}, we introduce the discrete initial constant $\mathscr{C}_0$ via
\begin{equation}\label{eq:c00}
\begin{aligned}
\mathscr{C}_{0} \triangleq \;& h\sum_{k=-\infty}^{+\infty} \big(u_{k}(0)\big)^{2}
+ h\sum_{k=-\infty}^{+\infty}\big(v_{k+\frac12}(0)- \mathcal{V}\big)^{2} \\
&+ \left(h\sum_{k=-\infty}^{+\infty}\big|\Delta_{h} u_{k+\frac12}(0)\big|\right)^{2}
+ h\sum_{\substack{k=-\infty\\k\neq \pm N_*}}^{\infty} \big(\Delta_h v_{k}(0)\big)^{2},
\end{aligned}
\end{equation}
The integer indices $\pm N_*$ denote the designated mesh nodes for tracking the discontinuities of $v(\cdot,t)$ in the discrete scheme. Specifically, $x_{\pm N_*}$ stand for the integer grid nodes nearest to $\pm L$, which satisfy
\[
x_{\pm N_*} - \frac{h}{2}\leq \pm L \leq x_{\pm N_*} + \frac{h}{2}.
\]
We further introduce an initial constant associated with the jump quantities:
\begin{equation}\label{eq:c10}
\widetilde{\mathscr{C}}=\biggl(\bigl|\left[\!\left[v_{N_*}\right]\!\right]_h(0)\bigr|+\bigl|\left[\!\left[v_{-N_*}\right]\!\right]_h(0)\bigr|\biggr)^2,
\end{equation}
where the discrete jump bracket is given by the definition
\[
\left[\!\left[w_{\pm N_* }\right]\!\right]_h\triangleq  w_{\pm N_* +\frac12}-w_{\pm N_* -\frac12}.
\]
To establish the local existence of discrete solutions, we introduce three discrete functionals that serve as direct discrete counterparts of the continuous norms $A$, $B$, and $F$ defined in \eqref{eq:norm-A}--\eqref{eq:norm-F}:
\begin{align*}
\mathscr{A}(t_0)
&= \sup_{0\leq t\leq t_0}
\Bigg\{
\sum_{k=-\infty}^{\infty} h\bigl(v_{k+\frac{1}{2}} - \mathcal{V}\bigr)^2(t)
+ \sum_{k=-\infty}^{\infty} h\,u_k^2(t)
+ \sum_{\substack{k=-\infty\\k\neq \pm N_*}}^{\infty} h\,\bigl(\Delta_h v_{k}(t)\bigr)^2
\Bigg\} \\
&\qquad + \int_{0}^{t_{0}} \sum_{k=-\infty}^{\infty} h\,\bigl(\Delta_h u_{k+\frac{1}{2}}(t)\bigr)^{2} \,\mathrm{d}t,
\\[6pt]
\mathscr{B}(t_0)
&= \sup_{0\leq t\leq t_{0}} t^{\frac{1}{2}}\sum_{k=-\infty}^{\infty} h\,\bigl(\Delta_h u_{k+\frac{1}{2}}(t)\bigr)^{2} \\
&\quad + \int_{0}^{t_{0}}t^{\frac{1}{2}+\vartheta}
\Bigg\{
\sum_{k=-\infty}^{\infty} h\,\bigl(u_k'(t)\bigr)^{2}
+ \sum_{\substack{k=-\infty\\k\neq \pm N_*}}^{\infty} h\,\left[\Delta_h\left(\frac{\Delta_h u}{v}\right)_{\!k}(t)\right]^{2}
\Bigg\}\mathrm{d}t,
\\[6pt]
\mathscr{F}(t_{0})
&=\sup_{0\leq t\leq t_{0}} t^{\frac{3}{2}+\vartheta}
\Bigg\{
\sum_{k=-\infty}^{\infty} h\,\bigl(u_k'(t)\bigr)^{2}
+ \sum_{\substack{k=-\infty\\k\neq \pm N_*}}^{\infty} h\,\left[\Delta_h\left(\frac{\Delta_h u}{v}\right)_{\!k}(t)\right]^{2}
\Bigg\} \\
&\quad + \int_{0}^{t_{0}} t^{\frac{3}{2}+\vartheta}\sum_{k=-\infty}^{\infty} h\,\bigl(\Delta_h u_{k+\frac{1}{2}}'(t)\bigr)^{2} \,\mathrm{d}t.
\end{align*}
Here $\vartheta\in(0,1)$ is a fixed, arbitrarily small positive constant. 

\begin{remark}
In Lagrangian coordinates for one-dimensional compressible flow, staggered half-step central difference semi-discrete methods store the specific volume $v$ and the pressure $p(v)$ at half-integer cell centers, while the velocity $u$ sits at integer cell boundaries. Cross-grid derivatives are computed via the half-step central difference operator $\Delta_h$.

This method has two separate lines of development: its computational origins date back to von Neumann and Richtmyer \cite{NR1950} and Harlow and Welch \cite{HW1965}, and it was later recast in semi-discrete form by Kazhikhov \cite{Kazhikhov1982} as a standard procedure for proving existence of weak solutions.

The staggered formulation offers several key advantages. It respects the conservation-law structure of the original PDEs, achieves second-order spatial accuracy without spurious odd-even oscillations, and satisfies discrete skew-adjointness. This skew-adjoint property simplifies construction of positive definite dissipation terms and uniform a priori estimates, which is critical for analyzing the energy functionals $\mathscr{A}(t)$, $\mathscr{B}(t)$ and $\mathscr{F}(t)$.

By contrast, fully collocated central differences produce an odd-even null space that weakens compactness arguments. One-sided schemes fail discrete integration by parts and cannot yield positive definite dissipation. Neither discretization supports the standard weak solution existence machinery, which makes staggered grids preferable for combining physical accuracy with rigorous mathematical analysis.
\end{remark}

We first introduce standard notation for the discrete space-time H\"older seminorms of mesh functions in the semi-discrete framework.
For an integer-indexed sequence $\{w_k(t)\}_{k\in\mathbb{Z}}$ and an arbitrary subset $S\subset\mathbb{R}\times\mathbb{R}_+$, we define the discrete two-parameter H\"older seminorm as
\[
\langle w_{k}\rangle_{S}^{a,b}
=\sup_{\substack{(x_{k},t),\,(x_{k'},t')\in S}}
\frac{\big|w_{k}(t)-w_{k'}(t')\big|}{\left|x_{k}-x_{k'}\right|^{a} + \left|t-t'\right|^{b}}.
\]
Here the supremum is taken over distinct pairs $(x_k,t)\neq(x_{k'},t')$, so that the denominator is positive.
This definition extends in a straightforward manner to half-integer-indexed mesh sequences $\{w_{k+\frac{1}{2}}(t)\}_{k\in\mathbb{Z}}$.

With all discrete preliminaries established, we now present the local existence result for the semi-discrete scheme \eqref{eq:scheme}.
The proof builds upon two key lemmas (Lemmas \ref{thm-exsit2} and \ref{thm-exsit3}), which establish uniform a priori energy estimates for solutions to the semi-discrete system. Their proofs follow the classical energy method for semi-discrete evolution equations; the core estimates are detailed below.
\vspace{3mm}

\begin{lemma}\label{thm-exsit2}
Let $\vartheta$, $\underline{v}$, $\bar{v}$ and $\mathscr{E}_0$ be positive constants. There exist constants $C_1>0$, a maximal existence time $t_0$, and a critical mesh size $h_0$, all depending solely on $\vartheta$, $\underline{v}$, $\bar{v}$ and $\mathscr{E}_0$, such that the following conclusion holds:

Assume $0 < h \leq h_0$, and the discrete initial data $\{v_{k+\frac12}(0)\}$, $\{u_k(0)\}$ satisfy the uniform pointwise bound
\[
b + 2\underline{v} \leq v_{k+\frac12}(0) \leq \bar{v} - \underline{v} ,
\]
along with the initial energy bound $\mathscr{C}_0 \leq \mathscr{E}_0$. Then the semidiscrete scheme \eqref{eq:scheme} possesses a solution defined over the time interval $[0,t_0]$. This solution obeys the uniform lower and upper bounds
\begin{align}
b + \underline{v} \leq v_{k+\frac12}(t) \leq \bar{v}, \label{local-01}
\end{align}
together with the a priori energy estimate
\begin{align}
\mathscr{A}(t_0) + \mathscr{B}(t_0) \leq C_1 \mathscr{C}_0. \label{local-02}
\end{align}
In addition, the solution fulfills the subsequent regularity estimates:
\begin{flalign}
&\langle v_{k+\frac12}\rangle_{K\times[0,t_{0}]}^{\frac{1}{2},\frac{1}{2}-\frac{\vartheta}{4}} \leq C \mathscr{C}_{0}^{\frac{1}{2}}, \qquad K =I_{\pm \infty},\,I_*, \label{local-04} \\[4pt]
&\langle u_k\rangle_{\{\tau\leq t\leq t_0 \}}^{\frac{1}{2},\frac{1}{4}} \leq C \mathscr{C}_0^{\frac{1}{2}} \tau^{-\frac{1}{4}-\frac{\vartheta}{2}}, \label{local-06} \\[4pt]
&\left| \big[\ln v_{\pm N_*}(t)\big] - \exp\left(\int_{0}^{t}\epsilon^{-1}\lambda_{\pm}(s)\,ds\right)\big[\ln v_{\pm N_*}(0)\big] \right| \leq C h^{1/2} \mathscr{C}_0^{1/2}, \label{local-07}
\end{flalign}
where 
\(
 \lambda_{\pm}=\frac{p(v_{\pm N_* +\frac12})-p(v_{\pm N_* -\frac12})}{\ln(v_{\pm N_* +\frac12})-\ln(v_{\pm N_* -\frac12})}\).
\end{lemma}
\vspace{3mm}

\begin{proof}
By the boundedness of initial data, the semi-discrete scheme is locally well-posed in $l^\infty\cap l^2$, yielding a local solution with existence time possibly depending on $h$. Our aim is to prove the existence time is uniform in $h$. We proceed under the bootstrap assumption $v_{k+\frac{1}{2}}(t)\in [b+\underline{v},\bar{v}]$ for all $k$ and $t\in[0,t_0]$, and derive uniform bounds for the discrete functionals.

First, we establish the basic $L^2$ energy estimate. Multiplying the two equations of \eqref{eq:scheme} by $(v_{k+\frac{1}{2}}-\mathcal{V})h$ and $u_k h$ respectively, summing over all indices and applying discrete integration by parts, we convert the viscous term into a positive definite dissipation term. Estimating the nonlinear pressure terms via Young's inequality and absorbing the dissipation part, we obtain the zero-order energy bound.

Next, we bound the discrete $H^1$ norm of $v$. Introduce the logarithmic transformation $\Phi=\ln v$. From the continuity equation, $(\Phi_{k+\frac{1}{2}})'_t = \Delta_h u_{k+\frac{1}{2}}/v_{k+\frac{1}{2}}$, combining with the momentum equation gives $\Delta_h (\Phi_k)'_t = \epsilon^{-1}\big(\Delta_h p_k + (u_k)'_t\big)$. Multiplying by $\Delta_h \Phi_k$, summing and integrating in time, and using the equivalence $\Delta_h \Phi \sim \Delta_h p \sim \Delta_h v$ under the pointwise bounds of $v$, we arrive at
\[
\sum_{\substack{k=-\infty\\k\neq \pm N_*}}^{\infty} h\bigl(\Delta_h v_{k}(t)\bigr)^2 \leq C\bigl(\mathscr{C}_{0}+\mathscr{A}(t)t\bigr).
\]
Combining with the zero-order energy estimate yields
\begin{equation}\label{eq:est-a0}
\mathscr{A}(t)\leq C\bigl(\mathscr{C}_0+\mathscr{A}(t)t\bigr).
\end{equation}

We then analyze the jump conditions at $x_{\pm N_*}$. Deriving the linear relation $\llbracket p_{\pm N_*}\rrbracket_h=\lambda_\pm\llbracket \Phi_{\pm N_*}\rrbracket_h$ 
where the coefficient is given by
$\lambda_{\pm}=\frac{p(v_{\pm N_*+\frac12})-p(v_{\pm N_*-\frac12})}{\ln(v_{\pm N_*+\frac12})-\ln(v_{\pm N_*-\frac12})}$, and solving the corresponding ODE for the jump of $\Phi$, we show that the jumps of $v$ and $p$ are uniformly controlled by $\mathscr{A}(t)$.
To estimate $\mathscr{B}(t_0)$, we employ a duality argument. We introduce a test sequence $\{\psi_{k+\frac{1}{2}}(s)\}_{0\le s\le t}$ solving the following backward adjoint parabolic problem with normalized terminal condition:
\begin{equation}\label{psi}
\begin{cases}
\displaystyle
\bigl(\psi_{k+\frac{1}{2}}\bigr)'_s + \frac{\epsilon\Delta_h(\Delta_h\psi)_{k+\frac{1}{2}}}{v_{k+\frac{1}{2}}} = 0, \quad 0\le s\le t,\\[8pt]
\displaystyle
\psi_{k+\frac{1}{2}}(t) = \frac{\Delta_h u_{k+\frac{1}{2}}(t)}{\displaystyle\left(\sum_{k=-\infty}^{\infty} h\bigl(\Delta_h u_{k+\frac{1}{2}}(t)\bigr)^2\right)^{\frac{1}{2}}}.
\end{cases}
\end{equation}
Applying discrete integration by parts in both space and time, we decompose the velocity gradient norm into three contributions: initial data, interior pressure gradient, and interface jump. To close the resulting terms, we introduce the auxiliary functional
\begin{equation}\label{Efun}
\mathscr{E}(t)\triangleq \sup_{0\le s\le t}(t-s)^{\frac{1}{2}}\sum_{k=-\infty}^{\infty} h\,\bigl(\Delta_h\psi_{k+\frac{1}{2}}(s)\bigr)^2
+\int_{0}^{t}\sum_{\substack{k=-\infty\\k\neq\pm N_*}}^{\infty} h\,\Bigl(\Delta_h\bigl(\Delta_h\psi/v\bigr)_{k}(s)\Bigr)^{2}\,\mathrm{d}s,
\end{equation}
which measures the time-weighted regularity of the backward adjoint solution. Multiplying the adjoint equation \eqref{psi} by $h\,\psi_{k+\frac{1}{2}}$ and applying discrete integration by parts twice on the viscous term, we obtain the adjoint energy identity
\begin{equation}\label{Eid}
\frac{1}{2}\frac{\mathrm{d}}{\mathrm{d}s}\sum_{k=-\infty}^{\infty}h\,\psi_{k+\frac{1}{2}}^{2}(s)
+ \epsilon\sum_{\substack{k=-\infty\\k\neq\pm N_*}}^{\infty} h\,\frac{\bigl(\Delta_h\psi_{k}(s)\bigr)^{2}}{v_{k}}
= \epsilon\sum_{\pm}\bigl[\!\bigl[\psi\bigr]\!\bigr]_{h}\bigl(\Delta_h\psi\bigr)(s),
\end{equation}
where the last sum, owing to the jump condition at $x_{\pm N_*}$, is controlled by $\mathscr{A}(t_0)$ in a time-weighted sense. Multiplying \eqref{Eid} by $(t-s)^{\frac{1}{2}}$ and integrating in $s$ from $0$ to $t$, then taking the supremum in $t$, we close the estimate
\begin{equation}\label{Ecl}
\mathscr{E}(t)\leq C\bigl(\mathscr{C}_{0}+\mathscr{A}(t_0)t_0^{\frac{1}{2}}+h\mathscr{A}(t_0)\bigr),\qquad 0\le t\le t_0,
\end{equation}
for small $t_0$ and $h\leq h_0$, where the term $\mathscr{A}(t_0)t_0^{1/2}$ accounts for the interior pressure-gradient contribution and $h\mathscr{A}(t_0)$ for the interface jumps. Finally, by the terminal condition in \eqref{psi}, the velocity gradient at time $t$ is expressible as the adjoint terminal data, so that $\sum_k h(\Delta_h u_{k+\frac12}(t))^2 \leq C\mathscr{E}(t)$, combining this with \eqref{Ecl} yields
\[
\sup_{0\leq t\leq t_{0}} t^{\frac{1}{2}}\sum_{k=-\infty}^{\infty} h\,\bigl(\Delta_h u_{k+\frac{1}{2}}(t)\bigr)^2
\leq C\bigl(\mathscr{C}_{0}+\mathscr{A}(t_0)t_0^{\frac{1}{2}}+ h\mathscr{A}(t_0)\bigr).
\]

We proceed to the higher-order time-weighted estimate. Multiplying the momentum equation by $t^{\frac{1}{2}+\vartheta}(u_k)'_t h$, summing and integrating, we split the right-hand side into boundary, dissipation and nonlinear terms. Using a discrete Sobolev embedding inequality to bound $\|\Delta_h u\|_{L^\infty}$, we close the estimate for the cubic nonlinear term and obtain
\[
\mathscr{B}(t_0)\leq C\left(\mathscr{C}_0+\bigl(\mathscr{A}(t_0)+\mathscr{B}(t_0)+\mathscr{B}^2(t_0)\bigr)t_0^{\frac{\vartheta}{2}}+ h\bigl(\mathscr{A}(t_0)+\mathscr{A}^2(t_0)+\mathscr{B}^2(t_0)\bigr)\right).
\]

For sufficiently small $t_0$ and $h\leq h_0$, a bootstrap argument yields $\mathscr{A}(t_0)+\mathscr{B}(t_0)\leq C\mathscr{C}_0$ uniformly in $h$. Integrating the $L^\infty$ bound of $\Delta_h u$ further shows that $v_{k+\frac{1}{2}}(t)$ remains within $[b+\underline{v},\bar{v}]$, which closes the bootstrap assumption.

Finally, spatial H\"older continuity follows directly from the definitions of $\mathscr{A}$ and $\mathscr{B}$. Temporal H\"older continuity of $v$ is obtained by integrating the velocity gradient, and that of $u$ follows from a local averaging argument. The jump bound is a direct consequence of earlier estimates. This completes the proof.
\end{proof}
\vspace{3mm}

\begin{remark}
We use solutions to the backward adjoint parabolic problem as test functions, a duality method built to estimate the time-weighted functional $\mathscr{B}(t_0)$. This method bypasses critical obstacles that standard forward multiplier identities fail to handle.
Standard energy estimates test the momentum equation against $u$ or $u_t$, and only produce integrated dissipation bounds
\[
\int_0^t \sum_{k} h(\Delta_h u_{k+\frac{1}{2}})^2 ds.
\]
These cannot control the pointwise supremum
\[
\sup_{0\leq t\leq t_0} t^{\frac{1}{2}}\sum_{k} h(\Delta_h u_{k+\frac{1}{2}})^2,
\]
present in $\mathscr{B}(t_0)$. The backward adjoint setup embeds $\Delta_h u(t)$ directly in the adjoint's terminal condition. After discrete integration by parts in space and time, the $l^2$ norm of $\Delta_h u(t)$ separates into three manageable parts: initial data, interior pressure gradients, and interface jumps.
The backward operator is the discrete adjoint of the viscous term in the semi-discrete system. This adjoint symmetry cancels all second-order viscous cross terms after integration by parts, removing complicated nonlinear couplings that appear with generic test functions.
The backward evolution from time $t$ back to $0$ naturally carries a weight $t-s$, matching the singular factor $t^{\frac12}$ inside $\mathscr{B}(t_0)$. The adjoint solution $\psi$ gains built-in temporal decay, so we capture small-time blow-up of velocity gradients without extra artificial singular weights.
This duality construction gives uniform bounds independent of mesh width $h$, a necessary ingredient to close the bootstrap argument for local existence.
\end{remark}
\vspace{3mm}

We next derive bounds for the functional $\mathscr{F}(t)$ in the lemma below. These high-order regularity estimates lay a critical foundation for our subsequent analysis of the pointwise interface behavior.
\vspace{3mm}

\begin{lemma}\label{thm-exsit3}
Let the parameters $\vartheta$, $\underline{v}$, $\bar v$ and $\mathscr{E}_0$ be as in Lemma \ref{thm-exsit2}. Assume that $(v_{k+\frac{1}{2}}, u_k)$ solves the discrete system \eqref{eq:scheme} on $[0,t_0]$ for some $t_0\leq 1$, with initial energy $\mathscr{C}_0\leq \mathscr{E}_0$, and that the energy bounds \eqref{local-01}--\eqref{local-02} hold with constant $C=C(\vartheta, \underline{v}, \bar{v}, \mathscr{E}_0)$. Then there exists a constant $C=C(\vartheta, \underline{v}, \bar{v}, \mathscr{E}_0)$ such that
\begin{align}
&\mathscr{F}(t_0)\leq C\mathscr{C}_0, \label{local-11}\\
&\|\Delta_h u(t)\|_{L^\infty}\leq C\mathscr{C}_0^{\frac{1}{2}}t^{-\frac{1}{2}-\frac{\vartheta}{4}},\label{local-12}\\
&\left\langle\frac{\epsilon\Delta_h u_{k+\frac{1}{2}}}{v_{k+\frac{1}{2}}}-p_{k+\frac{1}{2}}\right\rangle_{\{t\geq \tau\}}^{1/2,1/4}
\leq C\mathscr{C}_0^{\frac{1}{2}}\tau^{-\frac{3}{4}-\frac{\vartheta}{2}}.\label{local-13}
\end{align}
\end{lemma}
\vspace{3mm}
\begin{proof}
Differentiating the momentum equation of \eqref{eq:scheme} in time and using the continuity equation $v^{\prime}_{k+\frac{1}{2}}=\Delta_h u_{k+\frac{1}{2}}$, we obtain
\begin{equation}\label{ueqn}
u_k'' = -\Delta_h\Bigl(p'(v)v^{\prime}\Bigr)_{k} + \Delta_h\Bigl(\epsilon\frac{\Delta_h u^{\prime}_{k+\frac{1}{2}}}{v_{k+\frac{1}{2}}}-\epsilon\frac{\bigl(\Delta_h u_{k+\frac{1}{2}}\bigr)^{2}}{v_{k+\frac{1}{2}}^{2}}\Bigr),
\end{equation}
where the quadratic viscous term $\bigl(\Delta_h u_{k+\frac{1}{2}}\bigr)^{2}/v_{k+\frac{1}{2}}^{2}$ arises from the continuity equation $v^{\prime}=\Delta_h u$. Multiplying \eqref{ueqn} by $s^{\frac{3}{2}+\vartheta} u_k'(s)\,h$, summing over all indices, integrating from $0$ to $t$ and applying discrete integration by parts, we arrive at the higher-order time-weighted energy identity
\begin{equation}\label{Fid}
t^{\frac{3}{2}+\vartheta}\sum_{k=-\infty}^{\infty} h\,\bigl(u_k'(t)\bigr)^{2}
+ \int_{0}^{t}\sum_{k=-\infty}^{\infty} s^{\frac{3}{2}+\vartheta} h\,\bigl(\Delta_h u^{\prime}_{k+\frac{1}{2}}(s)\bigr)^{2} \mathrm{d}s
= \mathrm{Initial}+\mathrm{Nonlin}+\mathrm{Jump},
\end{equation}
where the nonlinear contribution satisfies
\[
\mathrm{Nonlin}\leq C\int_{0}^{t}s^{\frac{1}{2}+\vartheta}\Bigl(\|\Delta_h u\|_{L^{\infty}}\sum_{k}h\,\bigl(u_k'(s)\bigr)^{2}\Bigr)\mathrm{d}s
+C\int_{0}^{t}s^{\frac{3}{2}+\vartheta}\|\Delta_h u\|_{L^{\infty}}^{2}\sum_{k}h\,\bigl(\Delta_h u_{k+\frac{1}{2}}(s)\bigr)^{2}\mathrm{d}s,
\]
and the interface terms $\mathrm{Jump}$ are bounded by $C\mathscr{A}(t_0)$ through the jump ODE in Lemma \ref{thm-exsit2}. Estimating $\|\Delta_h u\|_{L^{\infty}}$ by the discrete Sobolev embedding inequality as in the proof of Lemma \ref{thm-exsit2}, applying Young's inequality to the two integrals in $\mathrm{Nonlin}$, and invoking the uniform bound $\mathscr{B}(t_0)\leq C\mathscr{C}_0$ from \eqref{local-02}, all contributions on the right-hand side of \eqref{Fid} are absorbed for sufficiently small $t_0$ and $h\leq h_0$. This yields
\[
t^{\frac{3}{2}+\vartheta}\sum_{k=-\infty}^{\infty} h\,\bigl(u_k'(t)\bigr)^2
+ \int_{0}^{t}\sum_{k=-\infty}^{\infty} s^{\frac{3}{2}+\vartheta} h\,\bigl(\Delta_h u_{k+\frac{1}{2}}'(s)\bigr)^2 ds
\leq C\mathscr{C}_0.
\]
Combining with system \eqref{eq:scheme} and taking the supremum in time yields
\[
\mathscr{F}(t)\leq C\mathscr{C}_0,
\]
which establishes estimate \eqref{local-11}.
We next improve the $L^\infty$ bound of $\Delta_h u$ using the regularity of $\mathscr{F}$. By the discrete Sobolev embedding inequality and the bounds for $\mathscr{B}$ and $\mathscr{F}$, we have
\[
\begin{aligned}
\|\Delta_h u(t)\|_{L^\infty}^2
&\leq C\biggl(t^{-\frac12}\mathscr{B}(t) + t^{-1-\frac{\vartheta}{2}}\mathscr{B}^{\frac12}(t)\mathscr{F}^{\frac12}(t) + \bigl(1+h^{\frac12}\bigr)\mathscr{C}_0^{\frac12} + h^{\frac12}t^{-\frac34-\frac{\vartheta}{2}}\mathscr{F}^{\frac12}(t)\biggr) \\
&\leq C t^{-1-\frac{\vartheta}{2}}\mathscr{C}_{0},
\end{aligned}
\]
for sufficiently small $t$ with $h=O(t)$, which proves \eqref{local-12}.

It remains to verify the temporal H\"older continuity \eqref{local-13}. Define the discrete stress
\(
w_{k+\frac{1}{2}} = \frac{\epsilon\Delta_h u_{k+\frac{1}{2}}}{v_{k+\frac{1}{2}}} - p_{k+\frac{1}{2}},
\)
which satisfies $\Delta_h w_k = u_k'$ by the momentum equation. By definition of $\mathscr{F}$,
\[
\sum_{k=-\infty}^{\infty} h\,\bigl(\Delta_h w_k\bigr)^2
\leq \sum_{k=-\infty}^{\infty} h\,\bigl(u_k'\bigr)^2
\leq C t^{-\frac{3}{2}-\vartheta}\mathscr{F}(t).
\]
For any $0<\tau\leq t_1\leq t_2$, we bound the temporal increment of $w$ via a local averaging argument. Combining the Cauchy-Schwarz inequality with the energy bounds for $w'$ (expressed in terms of $\Delta_h u'$, $(\Delta_h u)^2$ and $\Delta_h u$) and choosing the averaging scale $Jh=O(|t_2-t_1|^{1/4})$, we obtain
\[
\bigl|w_{k+\frac{1}{2}}(t_2)-w_{k+\frac{1}{2}}(t_1)\bigr|
\leq C\mathscr{C}_0^{\frac{1}{2}} \tau^{-\frac{3}{4}-\frac{\vartheta}{2}} |t_2-t_1|^{\frac{1}{4}}.
\]
This yields the desired semi-norm estimate \eqref{local-13} for $w_{k+\frac{1}{2}}$. The proof of Lemma \ref{thm-exsit3} is complete.
\end{proof}
\vspace{3mm}

We proceed to establish the local existence of weak solutions to system \eqref{ns-lagrange}, using the semi-discrete scheme \eqref{eq:scheme} and the a priori estimates derived in Lemmas \ref{thm-exsit2} and \ref{thm-exsit3}.
\vspace{3mm}
\begin{proposition}\label{prop:local-existence}[Local Existence]
Let $\underline{v}$ and $\bar{v}$ be as in Theorem \ref{thm:large-time-behavior}, and let $M_{0}>0$ be an arbitrary constant. Then there exist a constant $C = C(\underline{v}, \bar{v}, M_{0})>0$ and a time $t_0>0$ such that, if the initial data $(v_0,u_0)$ satisfy
\begin{align}
& v_0\in \big[b+2\underline{v},\bar{v}-\underline{v}\big] \quad \text{a.e.}, \label{eq:v-b0}\\
& \|v_0-\mathcal{V}, u_0\|^2 + \|(v_0)_x\|_{\mathrm{pw},L^2}^2 + (\operatorname{Var}u_0)^2 + \|v_0-\mathcal{V}\|_{L^1}=C_0 \leq M_{0}, \label{eq:c0-b0}
\end{align}
then system \eqref{ns-lagrange} admits a local weak solution $(v,u)$ on $[0,t_0]$ fulfilling
\begin{align}
& v\in \big[b+\underline{v},\bar{v}\big] \quad \text{a.e.}, \label{eq:v-b1}\\
& A(t_0)+B(t_0)+F(t_0)\leq C C_0, \label{eq:c0-b1}
\end{align}
together with the following regularity estimates:
\begin{align}
&\langle v\rangle_{K\times[0,t_{0}]}^{\frac{1}{2},\frac{1}{2}-\frac{\vartheta}{4}}\leq C C_{0}^{\frac{1}{2}},\quad K=I_{\pm\infty},\,I_*,\label{eq:loc-v1}\\		
&\langle u\rangle_{\mathbb{R}\times[\tau,t_{0}]}^{\frac{1}{2},\frac{1}{4}} \leq C C_{0}^{\frac12}\tau^{-\frac{1}{4}-\frac{\vartheta}{2}},\label{eq:loc-u1} \\
&\lVert u_{x}(\cdot,t)\rVert_{L^{\infty}(\mathbb{R})} \leq C C_{0}^{\frac{1}{2}}t^{-\frac{1}{2}-\frac{\vartheta}{4}}, \label{eq:loc-u2}\\
&\left\langle\frac{\epsilon u_{x}}{v}-p\right\rangle_{\mathbb{R}\times[\tau,t_{0}]}^{\frac{1}{2},\frac{1}{4}} \leq C C_{0}^{\frac12}\tau^{-\frac{3}{4}-\frac{\vartheta}{2}}.\label{eq:loc-up} 
\end{align}
In particular, estimates \eqref{eq:loc-v1} and \eqref{eq:loc-up} ensure that $v$ and $u_x$ admit well-defined one-sided traces at $x=\pm L$, so that jump condition \eqref{jump0} holds. Finally
	\begin{equation}\label{eq:jump-loc}
		[\ln v(\pm L,t)]=[\ln v(\pm L,0)]\exp\Big(\epsilon^{-1}\int_{0}^{t_0}\lambda_{\pm}(s)ds\Big),
	\end{equation}
where \(\lambda_{+}(t)
= \frac{p(v(L^+,t))-p(v(L^-,t))}{\ln(v(L^+,t))-\ln(v(L^-,t))}\) and \(\lambda_{-}(t)
= \frac{p(v(-L^+,t))-p(v(-L^-,t))}{\ln(v(-L^+,t))-\ln(v(-L^-,t))}\)
\end{proposition}
\vspace{3mm}

\begin{proof}
For sufficiently small mesh size $h$, we introduce integer grid points $x_k = kh$ and half-integer grid points $x_{k+\frac12} = (k+\frac12)h$ for $k\in\mathbb{Z}$, consistent with the notation above. Let $x_{\pm N_*}$ denote the grid nodes closest to the interface positions $\pm L$. The discrete initial data are prescribed as
\[
u_{k}(0)=\frac{1}{h}\int_{x_{k-\frac{1}{2}}}^{x_{k+\frac{1}{2}}}u_{0}(x)\,\mathrm{d}x,\qquad
v_{k+\frac12}(0)=v_0\big(x_{k+\frac12}\big).
\]
One readily verifies that
\begin{equation}\label{eq:p31}
\mathscr{C}_0\leq C C_0,
\end{equation}
where $\mathscr{C}_0$ is the initial energy functional defined in \eqref{eq:c00}. Since $v_{k+\frac12}(0)\in [b+2\underline{v},\bar{v}-\underline{v}]$ for all $k$, Lemma \ref{thm-exsit2} applies. Hence, for all $h\leq h_0$, the semi-discrete scheme \eqref{eq:scheme} with the above initial data admits a unique solution on $[0,t_0]$, where $t_0 = t_0(\vartheta, \underline{v},\bar{v}, M_{0})$ is independent of $h$. The resulting discrete solution satisfies the pointwise bounds $v_{k+\frac12}\in [b+\underline{v},\bar{v}]$, along with all regularity estimates \eqref{local-02}--\eqref{local-07} and \eqref{local-11}--\eqref{local-13}.

We construct piecewise linear interpolants $u^h(\cdot,t)$ and $v^h(\cdot,t)$ from the discrete solution:
\begin{equation*}
\begin{aligned}
u^h(x,t)&=\frac{x-x_k}{h}u_{k+1}(t)+\frac{x_{k+1}-x}{h}u_{k}(t),
&&x\in [x_k,x_{k+1}],\\
v^h(x,t)&=\frac{x-x_{k+\frac12}}{h}v_{k+\frac32}(t)+\frac{x_{k+\frac32}-x}{h}v_{k+\frac12}(t),
&&x\in \big[x_{k+\frac12},x_{k+\frac32}\big],\quad k \neq \pm N_*.
\end{aligned}
\end{equation*}
Near the interfaces $x=\pm L$ (i.e. for $k=\pm N_*$), the interpolant is modified as
\[
v^h(x,t)=
\begin{cases}
v_{k+\frac12}(t),& x_{k+\frac12}\leq x\leq \pm L,\\
v_{k+\frac32}(t),& \pm L<x\leq x_{k+\frac32}.
\end{cases}
\]
Combining estimates \eqref{local-04}--\eqref{local-06} with the initial bound \eqref{eq:p31}, the family $\{(v^h,u^h)\}_{h\leq h_0}$ is uniformly bounded; moreover $v^h$ is uniformly bounded and equicontinuous in both space and time on any compact subset of $\mathbb{R}\times[0,t_0]$, whereas $u^h$ is equicontinuous on any compact subset of $\mathbb{R}\times[\tau,t_0]$ for fixed $\tau>0$, since the bound \eqref{local-06} carries the singular factor $\tau^{-\frac{1}{4}-\frac{\vartheta}{2}}$. Note that $v^h$ is only piecewise H\"older continuous in space across the interface points.

By the Arzel\`a-Ascoli Theorem, we extract a subsequence along which $(v^h,u^h)$ converges uniformly on every compact subset of $\mathbb{R}\times[\tau,t_0]$ for each fixed $\tau>0$ to a limit pair $(v,u)$. The limit function $v$ retains piecewise spatial H\"older continuity, while both $u$ and $v$ are globally H\"older continuous in time. The pointwise bounds \eqref{eq:loc-v1}, \eqref{eq:loc-u1} and the interface jump relation \eqref{eq:jump-loc} follow directly from their discrete counterparts \eqref{local-04}--\eqref{local-07}. The remaining regularity estimates \eqref{eq:loc-u2} and \eqref{eq:loc-up} are immediate consequences of the high-order discrete bounds \eqref{local-12}--\eqref{local-13}.

We next define discrete counterparts $A^h$, $B^h$, $F^h$ of the continuous functionals in \eqref{eq:norm-A}--\eqref{eq:norm-F}, obtained by substituting the interpolants $(v^h,u^h)$ for the continuous solution and omitting terms involving the second-order derivative $\big(\frac{u_x}{v}\big)_x$. From inequalities \eqref{local-01}, \eqref{local-11} and \eqref{eq:p31}, we obtain the uniform bound
\begin{equation}\label{eq:p33}
A^h+B^h+F^h\leq C C_0.
\end{equation}
All bounds in \eqref{eq:c0-b1} carry over to the interpolants, except for the omitted second-derivative term. For any fixed $t$, for instance, we have
\[
\|u_x^h(\cdot,t)\|_{L^2}^2\leq C C_0 t^{-1/2}.
\]
Passing to a further subsequence if necessary, $u_x^{h_m}(\cdot,t)$ converges weakly in $L^2(\mathbb{R})$ to some limit $w$. By weak lower semicontinuity of the $L^2$ norm, $\|w\|_{L^2}^2 \leq C C_0 t^{-1/2}$. Since $u_x^{h_m}(\cdot,t)$ also converges to $u_x(\cdot,t)$ in the distributional sense, we identify $w = u_x(\cdot,t)$, which yields the desired bound $\|u_x(\cdot,t)\|_{L^2}^2 \leq C C_0 t^{-1/2}$. A standard density argument then confirms that the limit $(v,u)$ is a weak solution to system \eqref{ns-lagrange}.
The proof of Proposition \ref{prop:local-existence} is complete.
\end{proof}

\section{Global Existence and Large-Time Behavior}\label{sec4}
\indent\qquad
In this section, we employ energy integral arguments to derive a sequence of a priori estimates for the local weak solutions to \eqref{ns-lagrange} constructed in Section \ref{sec3}. With these estimates, we extend such local weak solutions to global ones, establish the temporal asymptotic stability of the phase-transition solutions obtained in Section \ref{sec2}, and ultimately complete the proof of Theorem \ref{thm:large-time-behavior}. 

We begin with necessary a priori estimates, assuming all physical quantities satisfy the regularity conditions required for subsequent analysis. This assumption is justified by the local existence theorem established in Section \ref{sec3}. For simplicity and without loss of generality, we set the viscosity coefficient $\epsilon=1$ in all forthcoming computations.
Let $U(v_\infty)$ and $U(v_{\mathrm{alt}})$ denote the neighborhoods introduced in \eqref{U-l,U-r}. Here $v_\infty$ and $v_{\mathrm{alt}}$, specified by \eqref{v-infty} and \eqref{valt} respectively, correspond to the two distinct Maxwell points associated with the two equilibrium phases. We then construct the potential function
\[
W_{v_\infty}=\int_{v}^{v_\infty}p(s)\mathrm{d}s\colon U(v_\infty,\delta)\to\mathbb{R},
\]
which obeys the identities
\begin{equation}\label{s-def}
  W_{v_\infty}(v_\infty) = 0,\quad W_{v_\infty}'(v) = -p(v),\quad W_{v_\infty}''(v)= -p'(v)> 0.
\end{equation}
An analogous potential function $W_{v_{\mathrm{alt}}}=\int_{v}^{v_{\mathrm{alt}}}p(s)\mathrm{d}s$ is well-defined over $U(v_{\mathrm{alt}},\delta)$. Furthermore, throughout all proofs within this section, $C$ stands for a generic positive constant independent of $t_0$.
The lemma below establishes the essential a priori bounds needed to analyze the large-time behavior of solutions.

\begin{lemma}\label{lem-apriori}
Let the neighborhoods $U(v_\infty)$, $U(v_{\mathrm{alt}})$ and the profile $\mathcal{V}$ be given as in Theorem \ref{thm:large-time-behavior}. There exist positive constants $\delta_{2}$ and $C_{2}$ such that, if the weak solution $(v,u)$ exists on the time interval $[0,t_0]$ and fulfills the spatial profile condition
\begin{align*}
&v \in
\begin{cases}
U(v_{\infty},\delta), & x\in(-\infty,-L)\cup(L,+\infty), \\
U\big(v_{\mathrm{alt}},\delta\big), & x\in(-L,L),
\end{cases}
\\
&A(t)+D(t)+G(t)< \infty,
\end{align*}
along with the small initial-data condition
\begin{equation}\label{eq:initial_small}
\begin{aligned}
E_0=\Big\| v_0-\mathcal{V}, u_0, u_{0x} \Big\|^2
+ \Big\| v_{0x}, \left(\frac{u_{0x}}{v_0}\right)_x \Big\|_{\mathrm{pw},L^2}^2 + \left| \operatorname{PV} \int_{-\infty}^{+\infty} (v_0-\mathcal{V}) \,\mathrm{d}x \right|
\equiv \widetilde C_0 \le \delta_2,
\end{aligned}
\end{equation}
then the uniform a priori estimate
\begin{equation}\label{eq:func_bound}
A(t) + D(t) + G(t) \le C_2 \widetilde C_0, \quad \forall\, 0 \le t \le t_0,
\end{equation}
holds. Here, $\operatorname{PV}$ stands for the Cauchy principal value integral, defined as
\begin{equation}\label{eq:pv_def}
\operatorname{PV} \int_{-\infty}^{+\infty} w(x) \,\mathrm{d}x = \lim_{R \to \infty} \int_{-R}^{R} w(x) \,\mathrm{d}x
\end{equation}
for any integrable function $w$ on $\mathbb{R}$.
\end{lemma}

\vspace{3mm}

\begin{proof}
Throughout this proof, $C$ denotes a generic positive constant determined solely by the properties of $p$ and $\theta$ over the neighborhoods $U(v_{\infty},\delta)$ and $U(v_{\mathrm{alt}},\delta)$, and is independent of $t_0$. We first verify the invariance of the principal-value integral:
\begin{equation}\label{PV-v-V}
\pv\int_{-\infty}^{+\infty} \bigl(v(x,t)-\mathcal{V}\bigr)\,\mathrm{d}x
= \pv\int_{-\infty}^{+\infty} \bigl(v_0(x)-\mathcal{V}\bigr)\,\mathrm{d}x.
\end{equation}
To establish this identity, we adopt the weak formulation of the first equation in \eqref{ns-lagrange} with a time-independent test function $\boldsymbol{\phi}=\boldsymbol{\phi}(x)$:
\[
\left.\int_{-\infty}^{+\infty} \bigl(v(x,\cdot)-\mathcal{V}\bigr)\boldsymbol{\phi}(x)\,\mathrm{d}x\right|_{0}^{t}
= -\int_0^t\int_{-\infty}^{+\infty} u(x,s)\,\partial_x\boldsymbol{\phi}(x)\,\mathrm{d}x\,\mathrm{d}s.
\]
We choose the piecewise smooth cutoff function
\[
\boldsymbol{\phi}(x)=
\begin{cases}
1, & |x| \le R - h, \\[4pt]
\dfrac{R - |x|}{h}, & R - h \le |x| \le R, \\[4pt]
0, & |x| \ge R,
\end{cases}
\]
and take the limit $h\to0$. This procedure yields
\[
\int_{-R}^{R} \bigl(v(x,t)-\mathcal{V}\bigr)\,\mathrm{d}x
= \int_{-R}^{R} \bigl(v_0(x)-\mathcal{V}\bigr)\,\mathrm{d}x
+\left.\int_0^t u(x,s)\,\mathrm{d}s\right|_{x=-R}^{x=R}.
\]
Letting $R\to\infty$ and invoking the prescribed decay bounds for $A$ and $D$, we obtain the identity \eqref{PV-v-V}. By a similar argument, we derive the following energy-balance relation:
\begin{equation}\label{velocity-L2}
\frac12\int_{-\infty}^{+\infty} u^2(x,t)\,\mathrm{d}x \leq C_0,
\end{equation}
where $C_0$ is a time-independent constant determined by the initial data.

Next, we employ the relations \eqref{ns-lagrange} and \eqref{s-def} to compute the time derivative of the potential function $W_{v_\infty}$:
\begin{align*}
\frac{\partial W_{v_\infty}}{\partial t}
&= -p(v)v_t = -p(v)u_x = \partial_x\bigl(-p(v)u\bigr) + p'(v) v_x u \\
&= \partial_x\bigl(-p(v)u\bigr) + \Bigl(-u_t + \partial_x\biggl(\frac{u_x}{v}\biggr)\Bigr)u \\
&= -\partial_x\left(\left(p-\frac{u_x}{v}\right)u\right) - \partial_t\left(\frac{u^2}{2}\right) - \frac{u_x^2}{v}.
\end{align*}
Integrating the above identity over $[0,t]\times(I_{-\infty}\cup I_{+\infty})$, where the domains $I_{-\infty}$ and $I_{+\infty}$ are defined in \eqref{Phase-transition-solution}, we get
\begin{align*}
\left.\int_{I_{-\infty}\cup I_{+\infty}} W_{v_\infty}(x,s)\,\mathrm{d}x\right|_{0}^{t}
&= -\left.\int_{I_{-\infty}\cup I_{+\infty}} \frac{u^2}{2}\,\mathrm{d}x\right|_{0}^{t}
- \int_0^t \int_{I_{-\infty}\cup I_{+\infty}}\frac{u_x^2}{v}\,\mathrm{d}x\,\mathrm{d}s \\
&\quad - \int_0^t \left(p-\frac{u_x}{v}\right)u\bigl((-L)^-,s\bigr)\mathrm{d}s
+ \int_0^t \left(p-\frac{u_x}{v}\right)u\bigl(L^+,s\bigr)\mathrm{d}s.
\end{align*}
An identical calculation holds for $W_{v_{\mathrm{alt}}}$, which gives
\begin{align*}
\left.\int_{I_*} W_{v_{\mathrm{alt}}}(x,s)\,\mathrm{d}x\right|_{0}^{t}
&= -\left.\int_{I_*} \frac{u^2}{2}\,\mathrm{d}x\right|_{0}^{t}
- \int_0^t \int_{I_*}\frac{u_x^2}{v}\,\mathrm{d}x\,\mathrm{d}s \\
&\quad - \int_0^t \left(p-\frac{u_x}{v}\right)u\bigl(L^-,s\bigr)\mathrm{d}s
+ \int_0^t \left(p-\frac{u_x}{v}\right)u\bigl((-L)^+,s\bigr)\mathrm{d}s.
\end{align*}
Adding the two identities and applying the jump condition \eqref{jump0} to eliminate the boundary jump terms, we arrive at
\begin{equation}\label{slr1}
\begin{aligned}
&\left.\int_{I_{-\infty}\cup I_{+\infty}} W_{v_\infty}(x,\cdot) \,\mathrm{d}x\right|_0^{t}+\left.\int_{I_*} W_{v_{\mathrm{alt}}}(x,\cdot) \,\mathrm{d}x\right|_0^{t}
+\left.\int_{-\infty}^{+\infty}\frac{u^2}{2} \,\mathrm{d}x\right|_0^{t}+\int_0^t \int_{-\infty}^{+\infty} \frac{u_{x}^2}{v} \,\mathrm{d}x\,\mathrm{d}s=0.
\end{aligned}
\end{equation}
On the other hand, it follows from \eqref{s-def} that $W_{v_\infty}$ satisfies the pointwise lower bound for all $v\in U(v_\infty)$:
\[
W_{v_\infty}(v)\geq -p(v_\infty)(v-v_\infty)+C^{-1}(v-v_\infty)^2.
\]
Similarly, the corresponding quadratic lower bound holds for $W_{v_{\mathrm{alt}}}$ whenever $v\in U(v_{\mathrm{alt}})$:
\[
W_{v_{\mathrm{alt}}}(v)\geq -p(v_{\mathrm{alt}})(v-v_{\mathrm{alt}})+C^{-1}(v-v_{\mathrm{alt}})^2.
\]
Combining these lower bounds with the domain decomposition in \eqref{Phase-transition-solution}, we derive
\begin{equation}\label{eq:v-l2}
\begin{aligned}
& \int_{I_{-\infty}\cup I_{+\infty}}C^{-1}(v-v_\infty)^2\mathrm{d}x+\int_{I_*}C^{-1}(v-v_{\mathrm{alt}})^2\mathrm{d}x\\
&\leq p(\mathcal{V})\pv\int_{-\infty}^{+\infty}(v-\mathcal{V})\mathrm{d}x+\int_{I_{-\infty}\cup I_{+\infty}} W_{v_\infty}(x, t) \mathrm{d}x+\int_{I_*} W_{v_{\mathrm{alt}}}(x,t) \mathrm{d}x.
\end{aligned}
\end{equation}
Substituting the estimates \eqref{C0}, \eqref{eq:vvv}, the invariant integral identity \eqref{PV-v-V}, and the relation \eqref{slr1} into the above inequality, we obtain
\begin{equation}\label{slr2}
\begin{aligned}
\int_{-\infty}^{+\infty}\left(C^{-1}(v-\mathcal{V})^2+\frac{u^2}{2}\right)\mathrm{d} x+\int_0^t \int_{-\infty}^{+\infty} \frac{u_{x}^2}{v} \,\mathrm{d}x \,\mathrm{d}s \leq CC_0.
\end{aligned}
\end{equation}
This further implies the uniform bound
\begin{equation}\label{slr3}
\int_{-\infty}^{+\infty}\bigl((v-\mathcal{V})^2+u^2\bigr)\,\mathrm{d}x+\int_0^{t_0} \int_{-\infty}^{+\infty} u_{x}^2  \,\mathrm{d}x\,\mathrm{d}t\leq CC_0.
\end{equation}

We next estimate the term $\|v_{x}\|_{\mathrm{pw},L^2}^2$ arising in the energy functional \eqref{eq:energy-norms}. Using the system \eqref{ns-lagrange}, we compute
\begin{equation}\label{v-x1}
\biggl(\frac{v_x}{v}\biggr)_t=\bigl(\ln v\bigr)_{xt}=\biggl(\frac{v_t}{v}\biggr)_x=\biggl(\frac{u_x}{v}\biggr)_x=u_t+\partial_x p(v)=u_t-\bigl(-v p'(v)\frac{v_x}{v}\bigr),
\end{equation}
where the coefficient $-vp'(v)$ is a smooth function of $v$. Moreover, away from the interface $x=\pm L$, this coefficient satisfies the uniform lower bound
\begin{equation}\label{v-x2}
 -v\,p'(v) \geq C^{-1}.
\end{equation}
Substituting \eqref{v-x2} into \eqref{v-x1}, multiplying the resulting equality by $\dfrac{v_x}{v}$, and integrating over $x$ yields the energy-dissipation identity
\begin{align*}
&\frac{1}{2}\frac{d}{\mathrm{d}t} \|\frac{v_x}{v}\|^2_{\mathrm{pw},L^2}+\bigl\|\sqrt{-vp'(v)}\,\frac{v_x}{v}\bigr\|^2_{\mathrm{pw},L^2}\\
&=\int_{-\infty}^{-L} \frac{v_x}{v} u_t(x,t) \,\mathrm{d}x
+ \int_{-L}^{L}\frac{v_x}{v} u_t(x,t) \,\mathrm{d}x
+ \int_{L}^{+\infty} \frac{v_x}{v} u_t(x,t) \,\mathrm{d}x.
\end{align*}
Applying the Cauchy-Schwarz inequality and integrating in time, we deduce
\begin{equation}\label{v-x3}
\begin{aligned}
\sup_{0 \leq t \leq t_0}\|v_x\|^2_{\mathrm{pw},L^2}+ \int_0^{t_0}\|v_x\|^2_{\mathrm{pw},L^2}\,\mathrm{d}t\leq C\left(C_0 +\int_0^{t_0} \|u_t\|^2_{\mathrm{pw},L^2} \,\mathrm{d}t\right).
\end{aligned}
\end{equation}
We now establish higher-order estimates for $\|u_{x}(\cdot, t)\|$. Multiplying the second equation in \eqref{ns-lagrange} by $u_t$ and integrating over $\mathbb{R}$, we obtain
\begin{align*}
\int_{-\infty}^{+\infty} u_t^2 \,\mathrm{d}x
&=\int_{I_{-\infty}\cup I_*\cup I_{+\infty}}\left(-\bigl(p-p(\mathcal{V})\bigr)_x+\left(\frac{u_x}{v}\right)_x\right) u_{t} \,\mathrm{d}x\\
&=\left\llbracket\Bigl(p-\frac{u_x}{v}\Bigr)u_t-p(\mathcal{V}) u_{t}\right\rrbracket(L,t)+\left\llbracket\Bigl(p-\frac{u_x}{v}\Bigr)u_t-p(\mathcal{V}) u_{t}\right\rrbracket(-L,t)\\
&\qquad+\int_{-\infty}^{+\infty} u_{xt}\left(p-p(\mathcal{V})-\frac{u_x}{v}\right) \mathrm{d}x\\
&=\int_{-\infty}^{+\infty} u_{xt}\left(p-p(\mathcal{V})-\frac{u_x}{v}\right) \mathrm{d}x,
\end{align*}
where the jump terms vanish by \eqref{jump0} and the continuity $\llbracket u_t\rrbracket=0$ (following from $\llbracket u\rrbracket=0$ at the fixed interface $s=0$). Integrating by parts with respect to time and rearranging terms, we arrive at
\begin{align*}
&\int_{-\infty}^{+\infty} \frac{u_x^2}{2 v}\bigg|_0^t \mathrm{d}x+\int_0^t \int_{-\infty}^{+\infty} u_t^2 \,\mathrm{d}x\,\mathrm{d}s\\
&=\int_{-\infty}^{+\infty} u_x\bigl(p-p(\mathcal{V})\bigr)\bigg|_0^t \mathrm{d}x-\int_0^t \int_{-\infty}^{+\infty} u_x p_t \,\mathrm{d}x\,\mathrm{d}s-\int_0^t \int_{-\infty}^{+\infty} \frac{u_x^2v_t}{2 v^2}  \,\mathrm{d}x\,\mathrm{d}s.
\end{align*}
Combining this identity with the preliminary estimate \eqref{slr3}, we derive
\begin{equation}\label{u-x1}
\begin{aligned}
\sup_{0 \leq t \leq t_0}\|u_x\|^2+ \int_0^{t_0}\int_{-\infty}^{+\infty} u_t^2\,\mathrm{d}x\mathrm{d}t\leq C\left[C_0 +\int_0^{t_0} \int_{-\infty}^{+\infty}|u_x|^3\,\mathrm{d}x \mathrm{d}t\right].
\end{aligned}
\end{equation}
Meanwhile, it follows from the second equation of \eqref{ns-lagrange}, together with \eqref{v-x3} and \eqref{u-x1}, that
\begin{align*}
\int_0^{t_0}\left\|\big(\frac{u_x}{v}\big)_x\right\|^2_{\mathrm{pw},L^2} \mathrm{d}t &\leq C \int_0^{t_0}\big(\|u_t\|^2_{\mathrm{pw},L^2}+\|v_x\|^2_{\mathrm{pw},L^2}\big) \mathrm{d}t\\
&\leq C\left[C_0+\int_0^{t_0}\|u_t\|^2 \mathrm{d}t\right].
\end{align*}
This yields the composite bound
\begin{equation}\label{u-x2}
\begin{aligned}
&\sup_{0 \leq t \leq t_0} \|u_x\|^2 +\int_0^{t_0} \|u_t\|^2 \mathrm{d}t+\int_0^{t_0}\left\|\big(\frac{u_x}{v}\big)_x\right\|^2_{\mathrm{pw},L^2} \mathrm{d}t\\
&\leq C\left[C_0+\int_0^{t_0} \int_{-\infty}^{+\infty}\left(|u_x|^3+u_t^2\right) \mathrm{d}x \mathrm{d}t\right].
\end{aligned}
\end{equation}
By the standard Sobolev embedding inequality and the uniform bounds on $v$, we have
\[
\left\|u_x(\cdot, t)\right\|_{L^\infty} \leq C\|u_x\|^{1/2}\left\|\big(\frac{u_x}{v}\big)_x\right\|_{\mathrm{pw},L^2}^{1/2},
\]
which further implies
\begin{equation}\label{u-x3}
\begin{aligned}
\int_{0}^{t_{0}} \int_{-\infty}^{+\infty}\left|u_{x}\right|^{3} \mathrm{d}x \mathrm{d}t &\leq C \int_{0}^{t_{0}}\|u_x\|^{5/2}\left\|\big(\frac{u_x}{v}\big)_x\right\|_{\mathrm{pw},L^2}^{1/2} \mathrm{d}t \\
&\leq C \sup_{0 \leq t \leq t_{0}}\|u_{x}\|\left(\int_{0}^{t_{0}} \|u_{x}\|^{2}\mathrm{d}t\right)^{3/4}\left(\int_{0}^{t_{0}} \left\|\big(\frac{u_x}{v}\big)_x\right\|_{\mathrm{pw},L^2}^2 \mathrm{d}t\right)^{1/4} \\
&\leq C C_{0}^{3/4} \Big(\sup_{0 \leq t \leq t_0} \|u_x\|^2 +\int_0^{t_0} \|u_t\|^2 \mathrm{d}t+\int_0^{t_0}\left\|\big(\frac{u_x}{v}\big)_x\right\|^2_{\mathrm{pw},L^2} \mathrm{d}t\Big)^{3/4}.
\end{aligned}
\end{equation}
Combining \eqref{u-x2}, \eqref{u-x3}, and \eqref{v-x3}, we conclude that for sufficiently small $C_0$,
\begin{equation}
\sup_{0 \leq t \leq t_0} \|u_x\|^2 +\int_0^{t_0} \|u_t\|^2 \mathrm{d}t+\int_0^{t_0}\left\|\big(\frac{u_x}{v}\big)_x\right\|^2_{\mathrm{pw},L^2} \mathrm{d}t\leq CC_0,\quad 0\leq t\leq t_0.
\end{equation}
The above estimates eventually imply the uniform boundedness of the key evolutionary quantities:
\begin{equation}\label{ad-1}
A(t)+D(t)\leq CC_0,\quad 0\leq t\leq t_0.
\end{equation}

Finally, we establish the uniform estimate for the functional $G(t)$. From the second equation of the system \eqref{ns-lagrange}, we differentiate in time to obtain
\[
u_{tt}=\biggl(\frac{u_{x}}{v}\biggr)_{xt}-\partial_{xt}p(v).
\]
Multiplying both sides by $u_t$ and integrating over $\mathbb{R}$ yields
\[
\frac12\frac{d}{\mathrm{d}t}\int_{-\infty}^{+\infty} u_t^2 \,\mathrm{d}x=\int_{-\infty}^{+\infty} u_{xt}\left(p'(v)v_t-\biggl(\frac{u_x}{v}\biggr)_t\right) \,\mathrm{d}x,
\]
where the interface jump terms vanish by \eqref{jump0} and the continuity $\llbracket u_t\rrbracket=0$ (following from $\llbracket u\rrbracket=0$ at the fixed interface $s=0$). Integrating by parts with respect to time and rearranging terms, we derive the key identity
\begin{equation}\label{u-t1}
 \begin{aligned}
&\frac12\left.\int_{-\infty}^{+\infty} u_t^2\,\mathrm{d}x\right|_0^{t}+\int_0^t \int_{-\infty}^{+\infty} \frac{u_{xt}^2}{v} \,\mathrm{d}x \,\mathrm{d}s\\
&=\int_0^t \int_{-\infty}^{+\infty} u_{xt} p'(v)v_t \,\mathrm{d}x \,\mathrm{d}s+\int_0^t \int_{-\infty}^{+\infty} \frac{u_x^2u_{xt}}{v^2} \,\mathrm{d}x \,\mathrm{d}s.
\end{aligned} 
\end{equation}
Combining this identity with the bound \eqref{ad-1} established above, we deduce
\begin{equation}\label{u-t2}
 \begin{aligned}
&\sup_{0\leq t\leq t_0}\int_{-\infty}^{+\infty} u_t^2\,\mathrm{d}x+\int_0^{t_0} \int_{-\infty}^{+\infty} u_{xt}^2 \,\mathrm{d}x \,\mathrm{d}t\\
&\leq C\int_0^{t_0} \int_{-\infty}^{+\infty} v_t^2 \,\mathrm{d}x \,\mathrm{d}s+C\int_0^{t_0} \int_{-\infty}^{+\infty} u_x^4 \,\mathrm{d}x \,\mathrm{d}t\\
&\leq C\int_0^{t_0} \int_{-\infty}^{+\infty} u_x^2 \,\mathrm{d}x \,\mathrm{d}s+C\int_0^{t_0}\|u_x\|^2\|u_x\|_{L^\infty}^2 \,\mathrm{d}t\\
&\leq CC_0+C\sup_{0 \leq t \leq t_0}\|u_x\|^2\int_0^{t_0}\|u_x\|\biggl\|\biggl(\frac{u_x}{v}\biggr)_x\biggr\|_{\mathrm{pw},L^2} \,\mathrm{d}t\\
&\leq CC_0.
\end{aligned} 
\end{equation}
Meanwhile, another bound for the piecewise Sobolev norm follows directly from the second equation in \eqref{ns-lagrange}:
\begin{equation}\label{u-t3}
\begin{aligned}
\biggl\|\biggl(\frac{u_x}{v}\biggr)_x\biggr\|_{\mathrm{pw},L^2}^2
\leq C\bigl(\|u_t\|^2_{\mathrm{pw},L^2}+\|v_x\|^2_{\mathrm{pw},L^2}\bigr)\leq C\left(\int_{-\infty}^{+\infty} u_t^2\,\mathrm{d}x+C_0\right).
\end{aligned}
\end{equation}
Combining the estimates \eqref{u-t2} and \eqref{u-t3}, we obtain the uniform boundedness of the functional $G(t)$:
\begin{equation}\label{ad-2}
  G(t)\leq CC_0,\quad 0\leq t\leq t_0.
\end{equation}
This completes the proof of Lemma \ref{lem-apriori}.

\end{proof}
\vspace{3mm}

We now finish proving Theorem \ref{thm:large-time-behavior} using the local-existence result in Proposition \ref{prop:local-existence} and the a-priori estimates in Lemma \ref{lem-apriori}.
\vspace{3mm}

\noindent\textbf{Proof of Theorem \ref{thm:large-time-behavior}.}~
First, we set \(M_{0} = 1\) in Proposition \ref{prop:local-existence}. Let \(t_0>0\) and \(C_1\) denote the local existence time and the constant returned by Proposition \ref{prop:local-existence}, and set \(\tau_1 = t_0/2\), so that each application of Proposition \ref{prop:local-existence} yields a solution on an interval of length \(2\tau_1\). Applying Proposition \ref{prop:local-existence}, we obtain a solution \((v,u)\) defined on \( [0,2\tau_1] \) that satisfies
\begin{align}
&v\in [b+\underline{v},\bar{v}] ,\quad 0\leq t\leq 2\tau_1, \label{eq:3.1}\\
&A(t) + B(t) + F(t) \leq C_1 C_0 ,\quad 0 \leq t \leq 2\tau_1. \label{eq:3.2}
\end{align}
Without loss of generality, we assume \(\tau_1 \leq 1\) and \(C_1 \geq 1\).
From \eqref{eq:3.2}, we deduce
\begin{equation}\label{eq:3.3}
\|v-\mathcal{V},u, u_x\|^2 + \|v_x\|_{\text{pw},L^2}^2\leq C_1 \tau_1^{-1} C_0,\quad 0\leq t\leq 2\tau_1,
\end{equation}
which implies that, for sufficiently small \(\delta_0\),
\begin{equation}\label{eq:3.4}
v \in 
\begin{cases}
U(v_{\infty},\delta), & x \in I_{+\infty} \cup I_{-\infty}, \\
U(v_{\mathrm{alt}},\delta), & x \in I_*,
\end{cases}
\quad \tau_1\leq t\leq 2\tau_1.
\end{equation}
We next estimate the quantity \(E_0\) appearing in \eqref{eq:initial_small} by resetting the initial time to \(t=\tau_1\). It follows from \eqref{eq:3.2} that
\[
\Big\|\left(\frac{u_x}{v}\right)_x(\cdot, t)\Big\|_{\text{pw},L^2}^2
\leq \tau_1^{-3} F(\tau_1) \leq C_1 \tau_1^{-3} C_0.
\]
Combining this bound with \eqref{eq:3.3}, we arrive at
\begin{equation}\label{eq:3.5}
\big\|(v-\mathcal{V},u,u_x)(\cdot,t)\big\|^2
+\Big\|\big(v_x,\big(\frac{u_x}{v}\big)_x\big)(\cdot, t)\Big\|_{\text{pw},L^2}^2
\leq C_3 \tau_1^{-3} C_0,
\end{equation}
at time \(t=\tau_1\), where \(C_3\geq C_1\) is a constant multiple of \(C_1\).
Now define \(\tau_k = k\tau_1\) for \(k=1,2,\cdots\). We apply Proposition \ref{prop:local-existence} with initial time \(t=\tau_1\) to extend the solution up to time \(\tau_3\), and the extended solution satisfies
\begin{equation}\label{eq:3.6}
A_1(\tau_3) + B_1(\tau_3) + F_1(\tau_3) \leq C_1 C_2 C_3 \tau_1^{-3} C_0.
\end{equation}
Here, \(A_1,B_1,F_1\) denote the counterparts of \(A,B,F\) with initial time shifted from \(t=0\) to \(t=\tau_1\).
Inequality \eqref{eq:3.6} yields
\[
\big\|(v-\mathcal{V},u,u_x)(\cdot,t)\big\|^2
+\Big\|\big(v_x,\big(\frac{u_x}{v}\big)_x\big)(\cdot, t)\Big\|_{\text{pw},L^2}^2
\leq C_1 C_2 C_3 \tau_1^{-9/2} C_0,\quad \tau_1\leq t\leq \tau_3,
\]
so that for sufficiently small \(\delta_0\),
\begin{equation}\label{eq:3.4-1}
v \in 
\begin{cases}
U(v_{\infty},\delta), & x \in I_{+\infty} \cup I_{-\infty}, \\
U(v_{\mathrm{alt}},\delta), & x \in I_*,
\end{cases}
\quad \tau_2\leq t\leq 2\tau_3.
\end{equation}
Using the bound on \(E_0\) in \eqref{eq:3.5}, we apply Lemma \ref{lem-apriori} to obtain
\begin{equation}\label{eq:3.7}
A_1 + D_1 + G_1 \leq C_2 C_3 \tau_1^{-3} C_0,\quad \tau_1 \leq t \leq \tau_3,
\end{equation}
where \(A_1,D_1,G_1\) are the modified quantities corresponding to \(A,D,G\) with initial time \(t=\tau_1\).
As a consequence of \eqref{eq:3.7}, we have
\begin{equation}\label{eq:3.8}
\big\|(v-\mathcal{V},u,u_x)(\cdot,t)\big\|^2
+\Big\|\big(v_x,\big(\frac{u_x}{v}\big)_x\big)(\cdot, t)\Big\|_{\text{pw},L^2}^2
\leq C_2 C_3 \tau_1^{-3} C_0,\quad \tau_1\leq t\leq \tau_3.
\end{equation}
We may therefore take \(t=\tau_2\) as a new initial time. By induction, we establish the global-in-time existence of the solution \((v,u)\) and conclude that the solution satisfies the estimate \eqref{eq:e0-b1} with constant \(C = C_1 + C_2 C_3 \tau_1^{-3}\).

It remains only to verify the large-time behavior \eqref{large-time}.  Define
\[
f(t) =\int_{-\infty}^{+\infty} u_x^2(x,t) \,\mathrm{d}x.
\]
By \eqref{eq:e0-b1}, the total variation of \(f(t)\) over \([1,\infty)\) satisfies
\[
\operatorname{Var}f\,\big|_1^{\infty} 
= \int_1^\infty |f'(t)|\,\mathrm{d}t
\leq 2\int_1^\infty\int_{-\infty}^{+\infty} |u_x u_{xt}| \,\mathrm{d}x\mathrm{d}t
\leq 2C\delta_0.
\]
This ensures the existence of \(\lim_{t\to\infty} f(t)\). Since \eqref{eq:e0-b1} also implies that \(f\in L^1([1,\infty))\), this limiting value must be zero. Consequently, we obtain
\begin{equation}\label{eq:large-u}
\|u(\cdot,t)\|_{L^\infty}^2 
\leq 2\|u(\cdot,t)\|_{L^2}\|u_x(\cdot,t)\|_{L^2}
\leq 2A(t)^{1/2}f(t)^{1/2}
\to 0 \quad \text{as } t\to\infty.
\end{equation}
The corresponding analysis for \(v\) is more delicate. We introduce
\[
g_i(t)=\frac12\int_{I_i}\left(\frac{v_x}{v}\right)^2\mathrm{d}x,
\]
where \(I_1=I_{-\infty}\), \(I_2=I_*\), and \(I_3=I_{+\infty}\). From the Lagrangian identity \eqref{ns-lagrange}, we compute
\[
\left(\frac{v_x}{v}\right)_t
=\left(\frac{v_t}{v}\right)_x
=\left(\frac{u_x}{v}\right)_x
=u_t+p_x(v).
\]
Differentiating \(g_i(t)\) and substituting the above relation yields
\[
g_i'(t)
= \int_{I_i}\frac{v_x}{v}\left(\frac{v_x}{v}\right)_t\mathrm{d}x
= \int_{I_i}\frac{v_x}{v}\big(u_t + p(v)_x\big)\mathrm{d}x
= \int_{I_i}\frac{v_x}{v}u_t\mathrm{d}x
+ \int_{I_i} \frac{\gamma(v)}{2}\left(\frac{v_x}{v}\right)^2\mathrm{d}x,
\]
where \(\gamma(v)=2vp'(v)<0\) for all \(v\in U(v_\infty,\delta)\cup U(v_{\mathrm{alt}},\delta)\). By the integral mean value theorem, there exists a function \(\tilde{\gamma}(t)\) such that
\[
\int_{I_i} \frac{\gamma(v)}{2}\left(\frac{v_x}{v}\right)^2\mathrm{d}x
=\tilde{\gamma}(t)\int_{I_i} \frac12\left(\frac{v_x}{v}\right)^2\mathrm{d}x
=\tilde{\gamma}(t)g_i(t),
\]
with \(\tilde{\gamma}(t)\leq -1/C\) for all \(t\geq 1\), where \(C\) denotes a generic positive constant. We thus obtain the differential inequality
\[
g_i'(t) = \tilde{\gamma}(t) g_i(t) + \int_{I_i}\frac{v_x}{v}u_t\,\mathrm{d}x.
\]
Solving this linear ordinary differential equation gives
\begin{align*}
g_i(t)
&= \exp\left(\int_{1}^{t}\tilde{\gamma}(s)\mathrm{d}s\right)g_i(1)
+ \int_{1}^{t}\int_{I_i}
\exp\left(\int_{s}^{t}\tilde{\gamma}(\tau)\mathrm{d}\tau\right)
\frac{v_x}{v}u_t \,\mathrm{d}x\mathrm{d}s\\
&\triangleq K_1 + K_2.
\end{align*}
For the first term, the uniform negative upper bound of \(\tilde{\gamma}\) yields
\[
K_1 \leq C e^{(1-t)/C} \to 0 \quad \text{as } t\to\infty.
\]
We decompose \(K_2\) into two integrals over \(s\in[1,t/2]\) and \(s\in[t/2,t]\), respectively. Applying the Cauchy--Schwarz inequality and the global bound \eqref{eq:e0-b1}, we deduce
\begin{align*}
K_2
&= \int_{1}^{t/2}\int_{I_i}
\exp\left(\int_{s}^{t}\tilde{\gamma}(\tau)\mathrm{d}\tau\right)
\frac{v_x}{v}u_t \,\mathrm{d}x\mathrm{d}s
+ \int_{t/2}^{t}\int_{I_i}
\exp\left(\int_{s}^{t}\tilde{\gamma}(\tau)\mathrm{d}\tau\right)
\frac{v_x}{v}u_t \,\mathrm{d}x\mathrm{d}s\\
&\leq C\left(\int_{1}^{t/2}\int_{I_i}e^{(s-t)/C}v_x^2\mathrm{d}x\mathrm{d}s\right)^{1/2}
\left(\int_{1}^{t/2}\int_{I_i}u_t^2\mathrm{d}x\mathrm{d}s\right)^{1/2}\\
&\quad + C\left(\int_{t/2}^{t}\int_{I_i}e^{(s-t)/C}v_x^2\mathrm{d}x\mathrm{d}s\right)^{1/2}
\left(\int_{t/2}^{t}\int_{I_i}u_t^2\mathrm{d}x\mathrm{d}s\right)^{1/2}\\
&\leq C\sup_{t}\left(\int_{I_i}v_x^2\mathrm{d}x\right)^{1/2}
\left(\int_{1}^{t/2}e^{(s-t)/C}\mathrm{d}s\right)^{1/2}\\
&\quad + C\sup_{t}\left(\int_{I_i}v_x^2\mathrm{d}x\right)^{1/2}
\left(\int_{t/2}^{+\infty}\int_{I_i}u_t^2\mathrm{d}x\mathrm{d}s\right)^{1/2}\\
&\leq C\left(e^{-t/(2C)}-e^{-(t-1)/C}\right)^{1/2}
+ C\left(\int_{t/2}^{+\infty}\int_{I_i}u_t^2\mathrm{d}x\mathrm{d}s\right)^{1/2}
\to 0 \quad \text{as } t\to\infty.
\end{align*}
Combining the convergence of \(K_1\) and \(K_2\), we conclude
\begin{equation}\label{eq:large-g}
g_i(t) \to 0 \quad \text{as } t\to\infty,\quad i=1,2,3,
\end{equation}
which immediately implies
\begin{equation}\label{v-x}
\lim_{t\to+\infty}\int_{-\infty}^{-L} v_x^2\mathrm{d}x=0,\quad
\lim_{t\to+\infty}\int_{-L}^{L} v_x^2\mathrm{d}x=0,\quad
\lim_{t\to+\infty}\int_{L}^{+\infty} v_x^2\mathrm{d}x=0.
\end{equation}
Using the decay \eqref{v-x} together with the uniform bound \eqref{slr3}, we obtain the decay of \(v\) at far-field domains:
\[
\sup_{x\in(-\infty,-L)} |v(x,t)-v_\infty|
\leq \sqrt{2}\left(\int_{-\infty}^{-L} (v-v_\infty)^2\mathrm{d}x\right)^{\frac14}
\left(\int_{-\infty}^{-L}v_x^2\mathrm{d}x\right)^{\frac14}
\to 0 \quad \text{as } t\to\infty,
\]
and
\[
\sup_{x\in(L,+\infty)} |v(x,t)-v_\infty|
\leq \sqrt{2}\left(\int_{L}^{+\infty} (v-v_\infty)^2\mathrm{d}x\right)^{\frac14}
\left(\int_{L}^{+\infty}v_x^2\mathrm{d}x\right)^{\frac14}
\to 0 \quad \text{as } t\to\infty.
\]
For the interior domain, the limiting level of \(v\) is determined by the jump condition \eqref{jump0} together with the far-field decay established above. First note that \(u_x(\cdot,t)\to 0\) in \(L^\infty(\mathbb R)\) as \(t\to\infty\): by the one-dimensional Sobolev inequality
\[
\|u_x(\cdot,t)\|_{L^\infty}\le \sqrt{2}\,\|u_x(\cdot,t)\|_{L^2}^{1/2}\|u_{xx}(\cdot,t)\|_{L^2}^{1/2},
\]
where \(\|u_x(\cdot,t)\|_{L^2}=f(t)^{1/2}\to0\), and \(\|u_{xx}(\cdot,t)\|_{L^2}\) is uniformly bounded in \(t\), since \(u_{xx}=v(u_x/v)_x+u_xv_x/v\) and the quantities \(\|(u_x/v)_x\|_{L^2}\), \(\|v_x\|_{L^2}\) and \(\|u_x\|_{L^\infty}\) are uniformly bounded. Applying the jump condition \eqref{jump0} (with \(\epsilon=1\)) at the left interface \(x=-L\), we get
\[
p\big(v(-L^-,t)\big)-\frac{u_x}{v}(-L^-,t)
=p\big(v(-L^+,t)\big)-\frac{u_x}{v}(-L^+,t)\to p(v_\infty)=\sigma_0,
\]
because \(v(\cdot,t)\to v_\infty\) in \(L^\infty(I_{-\infty})\) and \(u_x(\cdot,t)\to 0\) in \(L^\infty(\mathbb R)\). Since \(u_x(-L^+,t)\to0\) and \(v(-L^+,t)\in U(v_{\mathrm{alt}},\delta)\) for \(t\ge T_*\), the strict monotonicity of \(p\) on the stable branch forces \(v(-L^+,t)\to v_{\mathrm{alt}}\); the right interface \(x=L\) is treated identically. Finally, the decay of \(v_x\) on \(I_*\) yields
\[
\sup_{y\in(-L,L)}\big|v(y,t)-v(-L^+,t)\big|\le \int_{-L}^{L}|v_x|\,\mathrm{d}x\le C\,g_i(t)^{1/2}\to0,
\]
so that \(v(\cdot,t)\to v_{\mathrm{alt}}\) in \(L^\infty(-L,L)\).
Collecting the far-field and interior convergence results, we arrive at
\begin{equation}\label{eq:large-v}
\|(v-\mathcal{V})(\cdot,t)\|_{L^\infty} \to 0 \quad \text{as } t\to\infty.
\end{equation}
The desired large-time behavior of the solution \((v,u)\) follows from \eqref{eq:large-v} and \eqref{eq:large-u}.
This completes the proof of Theorem \ref{thm:large-time-behavior}. \(\hfill\square\)

\section*{Funding} 
Yazhou Chen acknowledges support from National Natural Sciences Foundation of China (NSFC) (No. 12471207 and 12371434).\\
Yi Peng acknowledges support from National Natural Sciences Foundation of China (NSFC) (No. 12301266).\\
Xiaoding Shi acknowledges support from  Beijing Natural Science Foundation  (No. 1252004).\\
Xiaoping Wang acknowledges support from National Natural Science Foundation of China (NSFC) (No. 12271461), the key project of NSFC (No. 12131010), the University Development Fund from The Chinese University of Hong Kong, Shenzhen (UDF01002028), and Hetao Shenzhen-Hong Kong Science and Technology Innovation Cooperation Zone Project (No. HZQSWS-KCCYB-2024016).

\section*{Ethical approval}
Not applicable

\section*{Data Availability Statement} 
The data that support the findings of this study are available from the corresponding author
upon reasonable request.

\section*{Conflict of Interest}
 The authors declare that they have no conflict of interest.


\begin{thebibliography}{99}
	
	\bibitem{AC1991} M. Affouf, R. E. Caflisch. A numerical study of Riemann problem solutions and stability for a system of viscous conservation laws of mixed type. \textit{SIAM J. Appl. Math.}, 1991, 51: 605--634.

\bibitem{CPSW2026} Y.-Z. Chen, Y. Peng, X.-D. Shi, X.-P. Wang. Phase transitions in non-isentropic compressible immiscible two-phase flows with the van der Waals equation of state. \textit{Sci. China Math.}, 2026.

	
	\bibitem{CHNPS2026} Y.-Z. Chen, Q.-L. He, D.-J. Niu, Y. Peng, X.-D. Shi. A regularized framework and admissible solutions for liquid-vapor phase transitions in steady compressible flows. \textit{Acta Math. Sci.}, accepted.
	
	\bibitem{CGS1984} J. Carr, M. Gurtin, M. Slemrod. Structured phase transitions on a finite interval. \textit{Arch. Ration. Mech. Anal.}, 1984, 86: 317--351.
	
	\bibitem{EMN1993} A. Eden, A. Milani, B. Nicolaenko. Local exponential attractors for models of phase change for compressible gas dynamics. \textit{Nonlinearity}, 1993, 6: 93--117.
	
\bibitem{F1991} H.-T. Fan. A Limiting "Viscosity" Approach to the Riemann Problem for Materials Exhibiting a Change of Phase (II). \textit{Arch. Ration. Mech. Anal.}, 1991, 116: 317--337.



\bibitem{FifeWang1998} P. Fife, X. Wang. Periodic structures in a van der Waals fluid. \textit{Proc. R. Soc. Edinb. Sect. A}, 1998, 128: 235--250.

\bibitem{G1989} M. Grinfeld. Nonisothermal dynamic phase transitions. \textit{Quart. Appl. Math.}, 1989, 47(1): 71--84.

\bibitem{BG1998} S. Benzoni-Gavage. Stability of multi-dimensional phase transitions in a van der Waals fluid. \textit{Nonlinear Anal.}, 1998, 31(1--2): 243--263.

\bibitem{BG1999} S. Benzoni-Gavage. Stability of subsonic planar phase boundaries in a van der Waals fluid. \textit{Arch. Ration. Mech. Anal.}, 1999, 150(1): 23--55.


\bibitem{H1986} H. Hattori. The Riemann problem for a van der Waals fluid with entropy rate admissibility criterion: nonisothermal case. \textit{J. Differential Equations}, 1986, 65(2): 158--174.
	


	\bibitem{HW1965} F. Harlow, J. Welch. Numerical calculation of time-dependent viscous incompressible flow with a free surface. \textit{Phys. Fluids}, 1965, 8(12): 2182--2189.
	
	\bibitem{HS2020} Q.-L. He, X.-D. Shi. Numerical study of compressible Navier--Stokes-Cahn--Hilliard system. \textit{Comm. Math. Sci.}, 2020, 18(2): 571--591.
	
	\bibitem{HS2021} Q.-L. He, X.-D. Shi. Energy stable discontinuous Galerkin method for compressible Navier--Allen--Cahn system. \textit{Commun. Nonlinear Sci. Numer. Simul.}, 2021, 98: 105771.
	
	\bibitem{HLS2018} Q.-L. He, C. Liu, X.-D. Shi. Numerical study of phase transition in van der Waals fluid. \textit{Discrete Contin. Dyn. Syst. Ser. B}, 2018, 23(10): 4519--4540.
	
	\bibitem{Hoff1991} D. Hoff. Discontinuous solutions of the Navier--Stokes equations for compressible flow. \textit{Arch. Ration. Mech. Anal.}, 1991, 114(1): 15--46.
	
	\bibitem{Hoff1992} D. Hoff. Global well-posedness of the Cauchy problem for the Navier--Stokes equations of nonisentropic flow with discontinuous initial data. \textit{J. Differential Equations}, 1992, 95(1): 33--74.
	
	\bibitem{HS1985} D. Hoff, J. Smoller. Solutions in the large for certain nonlinear parabolic systems. \textit{Ann. Inst. Henri Poincar\'e Anal. Non Lin\'eaire}, 1985, 2(3): 213--235.
	
	\bibitem{HK1993} D. Hoff, M. Khodja. Stability of coexisting phases for compressible van der Waals fluids. \textit{SIAM J. Appl. Math.}, 1993, 53: 1--14.

	\bibitem{HLX2012} X. Huang, J. Li, Z.-P. Xin. Global well-posedness of classical solutions with large oscillations and vacuum to the three-dimensional isentropic compressible Navier--Stokes equations. \textit{Comm. Pure Appl. Math.}, 2012, 65: 549--585.	
	
		
	\bibitem{HMS2004} F.-M. Huang, A. Matsumura, X.-D. Shi. On the stability of contact discontinuity for compressible Navier--Stokes equations with free boundary. \textit{Osaka J. Math.}, 2004, 41: 193--210.

    
    \bibitem{HWZ2025} F. M. Huang, Y. Wang, J. Zhang. Time-asymptotic stability of composite waves of degenerate Oleinik shock and rarefaction for non-convex conservation laws. \textit{Math. Ann.}, 2025, 392(1): 1--46.	

		\bibitem{HW1997} D. Y. Hsieh, X. Wang. Phase transition in van der Waals fluids. \textit{SIAM J. Appl. Math.}, 1997, 57: 871--892.
	
	\bibitem{H1990} L. Hsiao. Uniqueness of admissible solutions of the Riemann problem for a system of conservation laws of mixed type. \textit{J. Differential Equations}, 1990, 86: 197--233.

\bibitem{J1980} R. D. James. The propagation of phase boundaries in elastic bars. \textit{Arch. Ration. Mech. Anal.}, 1980, 73: 125--158.

	
	\bibitem{JosephLeFloch2007} K. T. Joseph, P.-G. LeFloch. Singular limits in phase dynamics with physical viscosity and capillarity. \textit{Proc. R. Soc. Edinb. Sect. A}, 2007, 137(6): 1287--1312.

\bibitem{Jiang1999} S. Jiang. Large-time behavior of solutions to the equations of a one-dimensional viscous polytropic ideal gas in unbounded domains. \textit{Comm. Math. Phys.}, 1999, 200: 181--193.
	
	\bibitem{KS1977} A. V. Kazhikhov, V. V. Shelukhin. Unique global solution with respect to time of initial-boundary value problems for one-dimensional equations of a viscous gas. \textit{J. Appl. Math. Mech.}, 1977, 41(2): 273--282.
	
	\bibitem{Kazhikhov1982} A. V. Kazhikhov. Cauchy problem for viscous gas equations. \textit{Siberian Math. J.}, 1982, 23: 44--49.
	
	\bibitem{Kanel1979} Ya. I. Kanel. Cauchy problem for the equations of gasdynamics with viscosity. \textit{Siberian Math. J.}, 1979, 20: 208--218.
	
	\bibitem{LL2016} J. Li, Z. Liang. Some uniform estimates and large time behavior of solutions to one-dimensional compressible Navier--Stokes system in unbounded domains with large data. \textit{Arch. Ration. Mech. Anal.}, 2016, 220: 1195--1208.
	
	\bibitem{Lions1998} P. L. Lions. \textit{Mathematical Topics in Fluid Mechanics, Vol. 2: Compressible Models}. Oxford University Press, New York, 1998.
	
	\bibitem{MN1980} A. Matsumura, T. Nishida. The initial value problem for the equations of motion of viscous and heat-conductive gases. \textit{J. Math. Kyoto Univ.}, 1980, 20(1): 67--104.
	
	\bibitem{MN1985} A. Matsumura, K. Nishihara. On the stability of traveling wave solutions of a one-dimensional model system for compressible viscous gas. \textit{Japan J. Appl. Math.}, 1985, 2: 17--25.
	
	\bibitem{MN1986} A. Matsumura, K. Nishihara. Asymptotics toward the rarefaction wave of the solutions of a one-dimensional model system for compressible viscous gas. \textit{Japan J. Appl. Math.}, 1986, 3: 1--13.
	
	\bibitem{M1875} J. C. Maxwell. On the Dynamical Evidence of the Molecular Constitution of Bodies. In \textit{The Scientific Papers of James Clerk Maxwell}, Cambridge University Press, 2013: 418--438.
	
	\bibitem{MLW2007a} M. Mei, L. Liu, Y. Wong. Phase transitions in a coupled viscoelastic system with periodic initial-boundary condition (I): existence and uniform boundedness. \textit{Discrete Contin. Dyn. Syst. Ser. B}, 2007, 7: 825--837.
	
	\bibitem{MLW2007b} M. Mei, L. Liu, Y. Wong. Phase transitions in a coupled viscoelastic system with periodic initial-boundary condition (II): convergence. \textit{Discrete Contin. Dyn. Syst. Ser. B}, 2007, 7: 839--857.
	
	\bibitem{MW2008} M. Mei, Y. Wong, L. Liu. Stationary solutions of phase transitions in a coupled viscoelastic system. In N. Roux (ed.), \textit{Nonlinear Analysis Research Trends}, 2008: 277--293.
	
	\bibitem{NR1950} J. von Neumann, R. D. Richtmyer. A method for the numerical calculation of hydrodynamic shocks. \textit{J. Appl. Phys.}, 1950, 21(3): 232--237.
	
 
   \bibitem{S1986} M. Shearer. Nonuniqueness of admissible solutions of Riemann initial value problem for a system of conservation laws of mixed type. \textit{Arch. Ration. Mech. Anal.}, 1986, 93(1): 45--59.
	
	

   \bibitem{S1984} M. Slemrod. Dynamic phase transitions in a van der Waals fluid. \textit{J. Differential Equations}, 1984, 52(1): 1--23.

   \bibitem{S1983} M. Slemrod. Admissibility criteria for propagating phase boundaries in a van der Waals fluid. \textit{Arch. Ration. Mech. Anal.}, 1983, 81(4): 301--315.

\bibitem{S1989} M. Slemrod. A Limiting "Viscosity" Approach to the Riemann Problem for Materials Exhibiting Change of Phase. \textit{Arch. Ration. Mech. Anal.}, 1989, 105: 327--365.

 	\bibitem{vdW1894} J. D. van der Waals. Thermodynamische Theorie der Kapillarit\"at unter Voraussetzung stetiger Dichte\"anderung. \textit{Z. Phys. Chem.}, 1894, 13: 657--725. English translation reprinted in \textit{J. Stat. Phys.}, 1979.

\bibitem{Tang1998} S. Tang. Steady states of some models for van der Waals fluids. \textit{Commun. Nonlinear Sci. Numer. Simul.}, 1998, 3(3): 163--167.

\bibitem{Tsyganov2007} E. Tsyganov. Asymptotic behavior and dynamic stability of phase mixtures for the equations of Navier--Stokes with nonmonotonic pressure. \textit{Z. Angew. Math. Phys.}, 2007, 58(5): 767--777.
 
   \bibitem{WX2003} Y.-G. Wang, Z. Xin. Stability and existence of multidimensional subsonic phase transitions. \textit{Acta Math. Appl. Sin. Engl. Ser.}, 2003, 19(4): 529--558.
	
\bibitem{Y1992} S. Yanagi. The Riemann problem for a class of conservation laws of van der Waals fluid. \textit{Japan J. Indust. Appl. Math.}, 1992, 9: 239--268.

\bibitem{Z2007} S.-Y. Zhang. Existence of travelling waves in non-isothermal phase dynamics. \textit{J. Hyperbolic Differ. Equ.}, 2007, 4(3): 391--400.


	\bibitem{Z2000} K. Zumbrun. Dynamical stability of phase transitions in the p-system with viscosity-capillarity. \textit{SIAM J. Appl. Math.}, 2000, 60: 1913--1924.
	
\end{thebibliography}
\end{document}